%% file: main_deliverable.tex
\documentclass[reqno,a4paper,12pt]{amsart}

\usepackage{microtype}
\usepackage[margin=2.6cm,bottom=1.25in]{geometry}		%

\usepackage[all,cmtip]{xy}

\usepackage{amsmath}		%
\usepackage{amssymb}		%
\usepackage{amsfonts}		%
\usepackage{amsthm}		%
\usepackage[foot]{amsaddr}		%

\usepackage{mathtools}     %
\mathtoolsset{%
}

\usepackage[utf8]{inputenc}%
\usepackage[T1]{fontenc}   %

\usepackage{libertine}
\usepackage[libertine,libaltvw,cmintegrals,varbb]{newtxmath}

\usepackage{dsfont}		%

\usepackage[%
cal=cm,
]
{mathalfa}

\usepackage[labelfont={bf,small},labelsep=colon,font=small]{caption}	%

\usepackage[dvipsnames,svgnames]{xcolor}						%
\colorlet{MyBlue}{DodgerBlue!75!Black}
\colorlet{MyGreen}{DarkGreen!85!Black}

\colorlet{MyBlue}{DodgerBlue!75!Black}
\colorlet{MyGreen}{DarkGreen!85!Black}
\colorlet{MyGray}{White!75!Black}

\colorlet{PrimalColor}{DodgerBlue!40!MidnightBlue}
\colorlet{DualColor}{FireBrick}

\definecolor{chartreuse}{HTML}{288800}
\definecolor{dark_gray}{HTML}{111111}
\definecolor{light_gray}{HTML}{555555}
\definecolor{tomato}{HTML}{FF6347}
\definecolor{snow}{HTML}{F1F1F1}
\definecolor{whiteee}{HTML}{EEEEEC}

\definecolor{c1}{RGB}{238,102,119}
\definecolor{c2}{RGB}{68, 119, 170}
\definecolor{c3}{RGB}{102, 204, 238}
\definecolor{c4}{RGB}{34, 136, 51}
\definecolor{c5}{RGB}{204, 187, 68}
\definecolor{c6}{RGB}{238, 102, 119} %
\definecolor{c7}{RGB}{170, 51, 119}
\definecolor{c8}{RGB}{187, 187, 187}

\definecolor{vibrant1}{HTML}{EE7733} %
\definecolor{vibrant2}{HTML}{0077BB} %
\definecolor{vibrant3}{HTML}{33BBEE} %
\definecolor{vibrant4}{HTML}{EE3377} %
\definecolor{vibrant5}{HTML}{CC3311} %
\definecolor{vibrant6}{HTML}{009988} %
\definecolor{vibrant7}{HTML}{BBBBBB} %

\definecolor{bright1}{HTML}{4477AA} %
\definecolor{bright2}{HTML}{EE6677} %
\definecolor{bright3}{HTML}{228833} %
\definecolor{bright4}{HTML}{CCBB44} %
\definecolor{bright5}{HTML}{66CCEE} %
\definecolor{bright6}{HTML}{AA3377} %
\definecolor{bright7}{HTML}{BBBBBB} %

\definecolor{light_blue}{HTML}{77AADD}
\definecolor{orange}{HTML}{EE8866}
\definecolor{light_yellow}{HTML}{EEDD88}
\definecolor{pink}{HTML}{FFAABB}
\definecolor{light_cyan}{HTML}{99DDFF}
\definecolor{mint}{HTML}{44BB99}
\definecolor{pear}{HTML}{BBCC33}
\definecolor{olive}{HTML}{AAAA00}
\definecolor{pale_grey}{HTML}{DDDDDD}
\definecolor{black}{HTML}{000000}

\usepackage{subcaption}
\usepackage{tikz}							%
\usetikzlibrary{calc,patterns}
\usetikzlibrary{arrows.meta}
\usetikzlibrary{external}

\usepackage{pgfplots}
\usepackage{pgfplotstable}
\pgfplotsset{compat=1.18}
\usepgfplotslibrary{groupplots}

\def\rot{\rotatebox} %
\usepackage{pifont}
\usepackage{multirow}
\newcommand{\cmark}{\ding{51}}%

\def\addlegendimage{\csname pgfplots@addlegendimage\endcsname}

\usepackage{array}		%
\usepackage{booktabs}		%
\usepackage[inline,shortlabels]{enumitem}		%

\usepackage{colortbl}

\newcommand{\citep}{\cite}
\newcommand{\citet}{\cite}

\usepackage{hyperref}
\hypersetup{
final,
colorlinks=true,
linktocpage=true,
pdfstartview=FitH,
breaklinks=true,
pdfpagemode=UseNone,
pageanchor=true,
pdfpagemode=UseOutlines,
plainpages=false,
bookmarksnumbered,
bookmarksopen=false,
bookmarksopenlevel=1,
hypertexnames=true,
pdfhighlight=/O,
urlcolor=Maroon,linkcolor=MyBlue!60!black,citecolor=DarkGreen!70!black,	%
pdftitle={},
pdfauthor={},
pdfsubject={},
pdfkeywords={},
pdfcreator={pdfLaTeX},
pdfproducer={LaTeX with hyperref}
}

\usepackage[backend=biber, style=numeric, sorting=nyt]{biblatex}
\renewbibmacro*{url+urldate}{}
\AtEveryBibitem{\clearfield{howpublished}}

\numberwithin{equation}{section}		%
\usepackage[sort&compress,capitalize,nameinlink]{cleveref}		%
\crefname{assumption}{Assumption}{Assumptions}
\crefrangeformat{equation}{\upshape(#3#1#4)\textendash(#5#2#6)}

\usepackage{algorithm}		%
\usepackage[noend]{algpseudocode}		%

\theoremstyle{plain}

\newtheorem{thm}{Theorem}[section]
\newtheorem{theorem}[thm]{Theorem}

\newtheorem{lemma}[thm]{Lemma}

\newtheorem{corollary}[thm]{Corollary}

\newtheorem{proposition}[thm]{Proposition}

\newtheorem*{claim*}{Claim}
\newtheorem*{tarskisProblem*}{Tarski's Problem (approximately)}
\newtheorem{smallsetstheorem*}[thm]{Small Sets Theorem}
\newtheorem{localOminimality*}[thm]{Local o-minimality}
\newtheorem{definableChoice*}[thm]{Definable Choice}
\newtheorem{dimensionTheorem*}[thm]{Dimension Theorem}
\newtheorem{tarskiSeidenberg*}[thm]{Tarski-Seidenberg Theorem}
\newtheorem{monotonicityTheorem*}[thm]{Monotonicity Theorem}
\newtheorem{smoothMonotonicityTheorem*}[thm]{Smooth Monotonicity Theorem}
\newtheorem{exponentialDichotomy*}[thm]{Exponential dichotomy}
\newtheorem{smallfrontiertheorem*}[thm]{Small Frontier Theorem}
\newtheorem{projectionFormula*}[thm]{Projection formula}
\newtheorem{chainRule*}[thm]{Chain rule}

\theoremstyle{definition}
\newtheorem{definition}[thm]{Definition}

\theoremstyle{remark}
\newtheorem{remark}[thm]{Remark}
\newtheorem{example}[thm]{Example}
\newtheorem{nonExample}[thm]{Non-example}

\newtheorem{openQuestion}[thm]{Open Question}

\crefname{property}{property}{properties}   %
\crefname{nonExample}{Non-Example}{Non-Examples}

\numberwithin{equation}{section}						%
\numberwithin{thm}{section}						%

\input{macros_deliverable}

\title{Deep learning as the disciplined construction of tame objects}

\author
[G.~Bareilles]
{Gilles Bareilles$^{1}$*}
\address{\scriptsize$^{1}$CMAP, École polytechnique, Institut Polytechnique de Paris, Palaiseau, France. Work done while at CTU.}
\email{gilles.bareilles@polytechnique.edu}

\author[A.~Gehret]{Allen Gehret$^{2, 3}$*}
\address{\scriptsize$^{2}$Czech Technical University in Prague, Artificial Intelligence Center, Charles Square 13, Prague 2, Czech Republic}
\address{\scriptsize$^{3}$Universit\"{a}t Wien, Institut f\"{u}r Mathematik, Kurt G\"{o}del Research Center, Kolingasse 14--16, 1090 Wien, Austria}
\email{gehreall@fel.cvut.cz}
\address{\scriptsize*Equal contribution}

\author
[J.~Aspman]
{Johannes Aspman$^2$}

\author
[J. Lep\v{s}ov\'a]
{Jana Lep\v{s}ov\'a$^2$}

\author
[J.~Mare\v{c}ek]
{Jakub Mare\v{c}ek$^2$}

\begin{document}

\begin{abstract}
  \input{abstract_deliverable.tex}

\end{abstract}

\maketitle

\input{body_deliverable.tex}

\printbibliography

\end{document}

%% file: macros_deliverable.tex
\usepackage[textwidth=30mm]{todonotes}		%

\newcommand{\debug}[1]{#1}		%

\usepackage[suppress]{color-edits}
\addauthor{gb}{DodgerBlue}
\addauthor{ja}{orange}
\addauthor{ag}{purple}

\newcommand{\newmacro}[2]{\newcommand{#1}{\debug{#2}}}		%
\newcommand{\newop}[2]{\DeclareMathOperator{#1}{\debug{#2}}}		%

\newmacro{\defeq}{\triangleq}
\newmacro{\eqdef}{\triangleq}

\DeclarePairedDelimiter{\clip}{[}{]}		%

\DeclarePairedDelimiterX{\setdef}[2]{\{}{\}}{#1:#2}		%
\DeclarePairedDelimiterXPP{\exclude}[1]{\mathopen{}\,\backslash}{\{}{\}}{}{#1}

\newcommand{\eg}{e.g.,}		%
\newcommand{\ie}{i.e.,}		%
\newcommand{\textpar}[1]{\textup(#1\textup)}		%

\newcommand{\alt}[1]{#1'}		%

\newcommand{\N}{\mathbb{N}}		%
\newcommand{\Q}{\mathbb{Q}}		%
\newcommand{\R}{\mathbb{R}}		%
\newcommand{\bbR}{\mathbb{R}}		%
\newcommand{\bbN}{\mathbb{N}}		%

\providecommand{\C}{C}

\newmacro{\Ostr}{\mathcal{R}}                      %
\newmacro{\Rsemialg}{$\Ostr_{\mathrm{alg}}$}                %
\newmacro{\Ran}{$\Ostr_{\mathrm{an}}$}        %
\newmacro{\Rexp}{$\Ostr_{\mathrm{exp}}$}      %
\newmacro{\Ranexp}{$\Ostr_{\mathrm{an,exp}}$} %
\newmacro{\Pff}{\mathrm{Pff}} %

\newmacro{\omin}{o-minimal} %
\newmacro{\Omin}{O-minimal} %

\newmacro{\dd}{\:d}		%
\newcommand{\eps}{\varepsilon}		%
\renewcommand{\epsilon}{\varepsilon}		%
\newop{\dom}{dom}		%

\newmacro{\set}{\mathcal{S}}		%

\newmacro{\points}{\mathcal{K}}		%
\newmacro{\point}{x}		%
\newmacro{\pointalt}{\alt\point}		%

\newmacro{\dpoints}{\mathcal{Y}}		%
\newmacro{\dpoint}{y}		%
\newmacro{\dpointalt}{\alt\dpoint}		%

\newmacro{\base}{p}		%
\newmacro{\basealt}{q}		%

\DeclareMathOperator{\dist}{dist}		%

\newmacro{\open}{\mathcal{U}}           %
\newmacro{\cpt}{\mathcal{K}}            %
\newmacro{\U}{U}                        %

\newmacro{\start}{1}		%
\newmacro{\running}{1,2,\dotsc}		%
\newmacro{\run}{t}		%
\newmacro{\runalt}{s}		%
\newmacro{\nRuns}{T}		%
\newmacro{\runtime}{S}		%

\newmacro{\runs}{\mathcal{\nRuns}}		%

\newmacro{\vecspace}{\mathcal{X}}		%
\newmacro{\vdim}{n}		%
\newmacro{\vvec}{v}		%
\newmacro{\bvec}{e}		%
\newmacro{\unitvec}{z}		%

\newmacro{\subspace}{\mathcal{W}}		%
\newmacro{\subdim}{m}		%

\newmacro{\tanspace}{\mathcal{Z}}		%
\newmacro{\tvec}{z}		%

\DeclarePairedDelimiterX{\braket}[2]{\langle}{\rangle}{#1,#2}		%

\newmacro{\dspace}{\vecspace^{\ast}}		%
\newmacro{\dvec}{w}		%
\newmacro{\dbvec}{\eps}		%

\newmacro{\ones}{\mathbf{1}}                     %
\newmacro{\mat}{A}                               %
\newmacro{\eye}{I}                               %

\newop{\Tr}{Tr}

\newmacro{\cvx}{\mathcal{C}}      %
\newmacro{\subd}{\partial}        %

\newmacro{\strong}{\alpha}        %
\newmacro{\smooth}{\beta}         %

\newmacro{\obj}{f}		%
\newmacro{\objalt}{g}		%
\newmacro{\sobj}{F}		%
\newmacro{\gvec}{g}		%
\newmacro{\gbound}{G}		%

\newmacro{\param}{\theta}		%
\newmacro{\params}{\Theta}		%

\newmacro{\oper}{A}		%
\newmacro{\vecfield}{V}		%
\newmacro{\vbound}{L}		%

\newmacro{\M}{M}		%
\newmacro{\Malt}{{\mathcal{M}}'}		%
\newmacro{\Mcol}{\mathcal{W}}		%

\newmacro{\X}{\mathcal{X}}		%
\newmacro{\Y}{\mathcal{Y}}		%

\newmacro{\vx}{x}

\newmacro{\vy}{y}
\newmacro{\vxalt}{x'}
\newmacro{\vyalt}{y'}

\newmacro{\vxman}{\vx^{\M}}

\newmacro{\vn}{v_n}
\newmacro{\vnalt}{\hat{v}_n}

\newmacro{\vsg}{v} %

\newmacro{\tangentBundle}{T\mathcal{B}}

\newcommand{\tangent}[2]{\debug{T}_{#2} (#1)}
\newcommand{\normal}[2]{\debug{N}_{#2} (#1)}
\newcommand{\tangentM}[1][\vx]{\tangent{#1}{\M}}
\newcommand{\normalM}[1][\vx]{\normal{#1}{\M}}

\newmacro{\maneq}{h}

\newmacro{\smoothcurve}{  c}
\newmacro{\geocurve}{  \gamma}

\newmacro{\grad}{\nabla_R}
\newmacro{\Hess}{  \operatorname{Hess}}

\newmacro{\D}{  \operatorname{D}}           %
\newmacro{\Jac}{  \operatorname{Jac}}           %

\newmacro{\fun}{F}   %
\newmacro{\funalt}{\tilde{F}}   %
\newmacro{\funman}{f} %
\newmacro{\funs}{f}   %
\newmacro{\funns}{f}  %
\newmacro{\funcomp}{F}  %
\newmacro{\mapping}{c}

\newop{\proj}{proj}
\newop{\prox}{prox}
\newop{\ri}{ri}
\newop{\Conv}{Conv}
\newop{\Aff}{Aff}
\newop{\Par}{Par}
\newop{\Diag}{Diag}
\newop{\diag}{diag}
\newop{\trace}{trace}
\newop{\bnd}{bnd}
\newop{\rbd}{rbd}

\newmacro{\ball}{{B}}

\newmacro{\ndim}{n}
\newmacro{\inputSpace}{\bbR^{\ndim}}

\newmacro{\ndimalt}{m}
\newmacro{\Rnalt}{\bbR^{\ndimalt}}
\newmacro{\interSpace}{\Rnalt}

\newmacro{\vxinter}{\vy}
\newmacro{\Minter}{\M^{\funns}}
\newmacro{\Mopt}{\opt[\M]}
\newmacro{\Mr}{\M^{\lambda_{\max}}_{\mult}}
\newmacro{\MI}{\M^{\max}_I}
\newmacro{\Mell}{\M_{I}^{\ell_{1}}}

\newmacro{\lammax}{\lambda_{\max{}}}
\newmacro{\mult}{r}

\newcommand{\opt}[1][\vx]{\debug{{#1}^\star}}		%
\newcommand{\crit}[1][\vx]{\debug{\bar{#1}}}				%

\newmacro{\const}{C}
\newmacro{\constalt}{C'}
\newmacro{\constcurve}{\tilde{L}}
\newmacro{\conststep}{C''}
\newmacro{\Tcurve}{T}

\newop{\Ima}{Im}

\DeclareMathOperator{\Kera}{Ker}
\renewcommand{\ker}{\Kera}

\newmacro{\cm}{ e}

\newmacro{\stepLowBnd}{\varphi}
\newmacro{\stepUpBnd}{\Gamma}
\usepackage{accents}                                  %
\newmacro{\stepLow}{\underaccent{\bar}{\step}}
\newmacro{\stepUp}{\bar{\step}}

\newmacro{\Mcodim}{p}
\newmacro{\maneqSpace}{\bbR^{\Mcodim}}

\newmacro{\dSQP}{d^{\mathrm{SQP}}}
\newmacro{\currdSQP}{d_{\ite}^{\mathrm{SQP}}(\curr[\vx])}

\newmacro{\dSQPt}{d^{\mathrm{SQP}}_{t}}
\newmacro{\dSQPn}{d^{\mathrm{SQP}}_{n}}
\newmacro{\dSQPr}{d^{\mathrm{SQP}}_{r}}

\newmacro{\gred}{\grad_{\M} \funcomp}
\newmacro{\redg}{\grad_{\M} \funcomp}
\newmacro{\redH}{\Hess_{\M} \funcomp}
\newmacro{\Z}{Z^{-}}
\newmacro{\Zt}{Z^{-\top}}

\newmacro{\hessLag}{\nabla^2_{xx} L}

\newmacro{\dcorr}{d^{\mathrm{corr}}}
\newmacro{\Lagmult}{\lambda}

\newmacro{\lsstep}{\alpha}
\newmacro{\lsArmijoParam}{m}

\newmacro{\Msqp}{M}

\newmacro{\stepinit}{\step_{\text{init}}}

\newmacro{\minmax}{\Phi}		%
\newmacro{\sadobj}{\minmax}		%

\newmacro{\minvar}{\point_{1}}		%
\newmacro{\minvaralt}{\alt\minvar}		%
\newmacro{\minvars}{\points_{1}}		%

\newmacro{\maxvar}{\point_{2}}		%
\newmacro{\maxvars}{\points_{2}}		%
\newmacro{\maxvaralt}{\alt\maxvar}		%

\newmacro{\play}{i}		%
\newmacro{\playalt}{j}		%
\newmacro{\nPlayers}{N}		%
\newmacro{\players}{\mathcal{\nPlayers}}		%

\newmacro{\pure}{a}		%
\newmacro{\purealt}{a'}		%
\newmacro{\nPures}{n}		%
\newmacro{\pures}{\mathcal{A}}		%

\newmacro{\cost}{c}		%
\newmacro{\loss}{\ell}		%
\newmacro{\pay}{u}		%
\newmacro{\payv}{v}		%
\newmacro{\pot}{F}		%

\newmacro{\game}{\mathcal{G}}		%
\newmacro{\gamefull}{\game(\players,\points,\pay)}		%

\newmacro{\fingame}{\Gamma}		%
\newmacro{\fingamefull}{\Gamma(\players,\pures,\pay)}		%

\newmacro{\mixgame}{\simplex(\fingame)}		%

\newmacro{\corstrat}{\pi}		%
\newmacro{\corprob}{\chi}		%
\newmacro{\cormarg}{\point}		%
\newmacro{\corprobs}{\points_{\mathrm{c}}}

\DeclareMathOperator{\ex}{\mathbb{E}}		%
\DeclareMathOperator{\prob}{\mathbb{P}}		%
\DeclareMathOperator{\simplex}{\Delta}		%

\newmacro{\sample}{\omega}		%
\newmacro{\samples}{\Omega}		%
\newmacro{\filter}{\mathcal{F}}		%
\newmacro{\probspace}{(\samples,\filter,\prob)}		%

\newmacro{\mean}{\mu}		%
\newmacro{\sdev}{\sigma}		%
\newmacro{\variance}{\sdev^{2}}		%

\newmacro{\dkl}{D_{\mathrm{KL}}}		%
\newmacro{\as}{\textpar{a.s.}\xspace}		%

\DeclarePairedDelimiterXPP{\exof}[1]{\ex}{[}{]}{}{%
 #1}

\DeclarePairedDelimiterXPP{\probof}[1]{\prob}{(}{)}{}{%
 #1}

\newmacro{\hreg}{h}		%
\newmacro{\breg}{D}		%
\newmacro{\proxmap}{P}		%
\newmacro{\mirror}{Q}		%
\newmacro{\fench}{F}		%
\newmacro{\hstr}{K}		%
\newmacro{\depth}{H}		%
\newmacro{\radius}{R}		%
\newmacro{\zone}{\mathbb{D}}		%
\newmacro{\subpoints}{\points^{\circ}}		%

\newmacro{\state}{X}		%
\newmacro{\dstate}{Y}		%
\newmacro{\signal}{V}		%

\newmacro{\step}{\gamma}		%
\newmacro{\learn}{\eta}		%

\newmacro{\ite}{k}
\newmacro{\initite}{0}
\newmacro{\afterinitite}{1}
\newmacro{\sumite}{n}
\newmacro{\Afterite}{K}

\newcommand{\curr}[1][\vx]{\debug{#1_{\ite}}}				%

\newmacro{\Graph}{\mathcal{G}}
\newmacro{\vertices}{\mathcal{V}}
\newmacro{\edges}{\mathcal{E}}

\newmacro{\ReLU}{\mathsf{ReLU}}		%

\newmacro{\smoothDeg}{p}

\newmacro{\stratDeg}{p}
\newmacro{\manDim}{d}
\newmacro{\spspDist}{\Delta}
\newmacro{\vecspDist}{\dist}

\newmacro{\fdef}{$f:\bbR^{\ndim}\to\bbR\cup\{+\infty\}$}
\newop{\epi}{epi} %
\newop{\gph}{gph} %

\newmacro{\AD}{\textbf{AD}}
\newmacro{\VC}{\text{VC}}
\colorlet{PrimalColor}{DodgerBlue!40!MidnightBlue}
\colorlet{DualColor}{FireBrick}

\definecolor{chartreuse}{HTML}{288800}
\definecolor{dark_gray}{HTML}{111111}
\definecolor{light_gray}{HTML}{555555}
\definecolor{tomato}{HTML}{FF6347}
\definecolor{snow}{HTML}{F1F1F1}
\definecolor{whiteee}{HTML}{EEEEEC}

\usepackage{xcolor}
\usetikzlibrary{patterns}

\definecolor{c1}{RGB}{238,102,119}
\definecolor{c2}{RGB}{40,90,140}
\definecolor{c3}{RGB}{0,120,170}
\definecolor{c4}{RGB}{0,110,100}
\definecolor{c5}{RGB}{20,110,40}
\definecolor{c6}{RGB}{120,150,40}
\definecolor{c7}{RGB}{200,160,40}
\definecolor{c8}{RGB}{210,120,30}
\definecolor{c9}{RGB}{140,40,110}
\definecolor{c10}{RGB}{100,20,70}
\definecolor{c11}{RGB}{70,70,70}
\definecolor{c12}{RGB}{120,120,120}
\definecolor{c13}{RGB}{40,30,110}

\definecolor{bright_blue}{HTML}{4477AA}   %
\definecolor{bright_red}{HTML}{EE6677}    %
\definecolor{bright_green}{HTML}{228833}    %
\definecolor{bright_yellow}{HTML}{CCBB44} %
\definecolor{bright_cyan}{HTML}{66CCEE}   %
\definecolor{bright_purple}{HTML}{AA3377} %
\definecolor{bright_grey}{HTML}{BBBBBB}   %

\definecolor{vibrant_orange}{HTML}{EE7733}  %
\definecolor{vibrant_blue}{HTML}{0077BB}    %
\definecolor{vibrant_cyan}{HTML}{33BBEE}    %
\definecolor{vibrant_magenta}{HTML}{EE3377} %
\definecolor{vibrant_red}{HTML}{CC3311}     %
\definecolor{vibrant_teal}{HTML}{009988}    %
\definecolor{vibrant_grey}{HTML}{BBBBBB}    %

\definecolor{light_blue}{HTML}{77AADD}
\definecolor{orange}{HTML}{EE8866}
\definecolor{light_yellow}{HTML}{EEDD88}
\definecolor{pink}{HTML}{FFAABB}
\definecolor{light_cyan}{HTML}{99DDFF}
\definecolor{mint}{HTML}{44BB99}
\definecolor{pear}{HTML}{BBCC33}
\definecolor{olive}{HTML}{AAAA00}
\definecolor{pale_grey}{HTML}{DDDDDD}
\definecolor{black}{HTML}{000000}

\newmacro{\Whcusp}{W^\text{cusp}}

\newmacro{\fnegabs}{f^{1}}
\newmacro{\fsubfail}{f^{2}}

\newcommand{\proofof}[1]{\emph{Proof of #1.}}

\newmacro{\xa}{x'}
\newmacro{\Xa}{\X'}
\newmacro{\ya}{y'}
\newmacro{\Ya}{\Y'}

\newmacro{\stra}{strong-$(a)$}

\newmacro{\nsfun}{f}

\newmacro{\Cell}{\mathsf{Cell}^r}
\newmacro{\Adj}{\mathsf{Adj}^r}

\newmacro{\wHa}{\mathsf{wH}_{a}}
\newmacro{\wHagen}{\mathsf{wH}_{a}^{\operatorname{gen}}}

\newmacro{\vEr}{\mathsf{vEr}}
\newmacro{\vErgen}{\mathsf{vEr}^{\operatorname{gen}}}

\newmacro{\stratif}{\mathfrak{S}}
\newmacro{\linearmap}{A}

%% file: abstract_deliverable.tex
One can see deep-learning models as compositions of functions within the so-called tame geometry.
In this expository note, we give an overview of some topics at the interface of \emph{tame geometry} (also known as \emph{o-minimality}), \emph{optimization theory}, and \emph{deep learning} theory and practice.
To do so, we gradually introduce the concepts and tools used to build \emph{convergence guarantees} for stochastic gradient descent in a general nonsmooth nonconvex, but tame, setting.
This illustrates some ways in which tame geometry is a natural mathematical framework for the study of AI systems, especially within Deep Learning. %

%% file: body_deliverable.tex
\section{Introduction}

\subsection{Context: Frameworks for studying AI systems and Deep Learning}

The rapid evolution of Artificial Intelligence (AI) systems has allowed these systems to be applied across many fields, including high-risk applications such as credit-scoring \cite{bhatoreMachineLearningTechniques2020}, recidivism forecasting \cite{angwin2022machine}, and self-driving vehicles \cite{bachuteAutonomousDrivingArchitectures2021}.
While these advances promise social benefits, they also raise pressing concerns regarding the reliability, interpretability, fairness, and safety of systems that are both novel and deployed at scale.
Consequently, there is a growing consensus among scholars, policymakers, and industry stakeholders for the establishment of robust regulatory frameworks and standardized evaluation protocols to ensure a responsible and trustworthy AI deployment  \cite{dilhacMontrealDeclarationResponsible2018,europeanparliamentArtificialIntelligenceAct2024}.
To that end, a number of road-maps have been suggested in order to align AI systems with human preferences; see \eg{} \cite{russellHumanCompatibleArtificial2019,hoangFabuleuxChantierRendre2021} among others.

\medskip\noindent
This prompts the question: what is a good theoretical framework to provide guarantees on current AI systems, or more precisely, on Deep Learning objects?
Let us be specific: by ``good framework'', we mean a framework that is both \emph{realistic}, in that it encompasses relevant applications, such as current AI systems, and \emph{prolific}, in that it allows one to develop nontrivial theoretical guarantees that are relevant in practice.

\medskip\noindent
One such widespread framework is convex analysis and close variants \cite{nesterovLecturesConvexOptimization2018,hiriart-urrutyFundamentalsConvexAnalysis2001,rockafellar2009variational}.
Within convex analysis, numerous optimization schemes have been designed and analyzed, in terms of where they converge to, and at what speed; see \eg{} \cite{bottouOptimizationMethodsLargeScale2018,bachLearningTheoryFirst2024}.
Yet, these theories reach their limits with deep learning theory, where objects typically do not satisfy convexity, or close variants at every point.

\medskip\noindent
In this paper, we discuss how an interesting candidate stems from the interface of \emph{tame geometry}, or more precisely, \emph{o-minimality},\footnote{Here, we will not distinguish between \emph{tame geometry} and \emph{o-minimality}, we will often abbreviate \emph{definable in an o-minimal structure} to just \emph{definable}, and moreover use \emph{tame} as a synonym for \emph{definable} (in an o-minimal structure); we clarify these conventions in \Cref{sec:defomin}. The interested reader should be aware that this usage is part of a wider usage of \emph{tame} in the sense of Shelah's \emph{dividing lines} in model theory~\cite{shelah1990classification,forkingAndDividing} as well as a similar usage of \emph{tame} in the sense of \emph{tame expansions of the real field} in the work of Miller, Hieronymi and others~\cite{miller2005tameness,Hieronymi2025Tameness}.
} with \emph{optimization theory}, and \emph{deep learning}, both in terms of how realistic and prolific they are.
Let us detail these two aspects.

\subsection{A realistic framework}

Definability in o-minimal structures covers nearly all functions and problems occurring in Deep Learning theory and practice.
It does so by providing \emph{composability guarantees}: the objects built from definable \emph{elementary constituents} according to certain \emph{composition rules} will remain definable.

\medskip\noindent
The list of \emph{composition rules} that preserve definability is wide.
To quote
Ioffe~\cite{Ioffe08},
\begin{quote}
  \emph{...the class [of definable sets and functions] is closed with respect to practically all operations used in variational analysis...}
\end{quote}
In less technical terms, definability is preserved by almost all common operations in optimization, such as composition of functions, (partial) minimization, construction of dual objects, (sub)differentiation, \emph{etc}.
This feature follows directly from the first-order logic background of such structures, as discussed in \Cref{sec:semialg,sec:defomin}.
One notable exception is that solutions to differential equations may not remain definable in some o-minimal structure.

\medskip\noindent
The list of \emph{elementary constituents} which are definable in o-minimal structures is also wide.
In particular, it involves nearly all activation functions and losses that are commonly used in Deep Learning.
\Cref{table:definable} provides an overview of common activation functions and losses, and whether they are definable in three specific o-minimal structures, $\R_{\operatorname{alg}}$, $\R_{\operatorname{exp}}$, and $\R_{\operatorname{Pfaff}}$, defined in \Cref{sec:examples}.
\Cref{400_AFs_remark} provides a more in-depth discussion on the definability of activation functions in Deep Learning.

\medskip\noindent
For these two reasons, nearly every object that appears in Deep Learning theory is definable in some o-minimal structure.
This observation has been made a number of times, see \eg{} \cite{tressl2010introduction,davis2020stochastic} to cite just a few.
Note that many of the popular theoretical frameworks used to analyze Deep Learning objects fail to account for feedforward ReLU networks, arguably the simplest networks.
Indeed, the vast majority of approaches developed to study optimization schemes require the functions to be either differentiable with a (globally) Lipschitz gradient, or to feature some form of convexity such as strong convexity, convexity, weak convexity, prox-regularity, or at the very least Clarke regularity, in order of decreasing strength; see \eg{} \cite{rockafellar2009variational,clarkeOptimizationNonsmoothAnalysis1990}.
Yet, all these assumptions fail for the function $x \mapsto -\operatorname{ReLU}(x) = \min(0, -x)$ at zero.
This function can be seen as a ReLU Network with depth one; it reveals the nature of the difficulty of nonsmooth nonconvex optimization.
These nonsmooth nonconvex features are also present in higher-dimensional and more complex networks; yet they are (nearly) always definable in some o-minimal structure.

\medskip\noindent
Note that many theoretical frameworks come with \emph{composability guarantees}.
One popular such framework is convexity theory; for instance, adding two convex functions, or taking the indicator function of a convex set always maintains convexity, while subtracting two convex functions does not.
Here again, building functions and sets from convex objects and according to a list of prescribed composition rules, in a ``disciplined way'', is useful.
It provides objects that are convex,
and which are in a form suitable for optimization; see the disciplined convex programming line of work \cite{grantDisciplinedConvexProgramming2006}.

\medskip\noindent
To fix ideas, let us consider the example of training a simple Deep Neural Network.
Setting the stage, assume that a dataset of $N$ observations $D = \{ (x_i, y_i), i = 1, ..., N \}$ is available, and consider the task of finding a function $f_{\theta}$ such that $f_{\theta}(x_{i}) \approx y_{i}$.
The function $f_\theta$ is a \emph{Deep Neural Network} (DNN) of depth $L$ defined recursively, for some input $x$, as
\begin{equation}\label{eq:defDNN}
  a_0=x, \qquad a_i=\rho_i(V_i(\theta)a_{i-1}) ~~\forall i=1,\ldots,L,\qquad f_{\theta}(x)= a_{L},
\end{equation}
where
$V_i(\cdot)$ are linear maps into the space of matrices, and
$\rho_i$ are activation functions applied coordinate-wise.
The learning task, known as empirical risk minimization (ERM), is typically formulated as solving
\begin{equation}\label{eq:learnpb}\tag{ERM}
    \min_{\theta\in\bbR^n} \frac{1}{N} \sum_{i=1}^N \ell( f_\theta(x_i) , y_i),
\end{equation}
where
$\ell(\cdot,\cdot)$ is a loss function.
\Cref{table:definable} lists common activation and loss functions.
\Cref{eq:defDNN,eq:learnpb} illustrate that DNNs and the corresponding learning problem are obtained by minimization and composition of elementary constituents, including activation functions, loss functions, and operations such as sum and product.

\begin{table}[t]
  \centering
  \caption{Commonly used loss and activation functions and their
    definability in the o-minimal structures $\R_{\operatorname{alg}}$,
    $\R_{\operatorname{exp}}$ and $\R_{\operatorname{Pfaff}}$, defined in \Cref{sec:examples}.
    Here $\beta > 0$, $\operatorname{erf}(t) = \tfrac{2}{\sqrt{\pi}}\int_{0}^{t} e^{-s^2} ds$
    is the Gaussian error function.
    The symbol (\cmark) indicates that the property is inherited from the inclusion $\R_{\operatorname{alg}}\subset \R_{\operatorname{exp}}\subset\R_{\operatorname{Pfaff}}$. See \Cref{rem:table_semialgebraic}, \Cref{cor:table_exp} and \Cref{cor:table_pfaff}
    for details.
}
  \label{table:definable}
  \small
  \resizebox{1.0\textwidth}{!}{
    \begin{tabular}{cll|ccc}
      \toprule
      & Name & Mathematical formulation & $\R_{\operatorname{alg}}$ & $\R_{\operatorname{exp}}$ & $\R_{\operatorname{Pfaff}}$ \\
      \midrule
      \multirow{10}{*}{\rot{90}{Activation function}}
      & ReLU & $\rho(t) = \max(0,t)$ & \cmark & (\cmark) & (\cmark) \\
      & Softsign & $\rho(t) = \frac{t}{|t|+1}$ & \cmark & (\cmark) & (\cmark) \\
      & Logistic & $\rho(t) = \frac{1}{1+e^{-t}}$ & - & \cmark & (\cmark) \\
      & Hyperbolic tangent & $\rho(t) = \tanh(t)$ & - & \cmark & (\cmark) \\
      & Softplus & $\rho(t) = \log(1+e^t)$ & - & \cmark & (\cmark) \\
      & Swish & $\rho(t) = \tfrac{t}{1+e^{-\beta t}}$ & - & \cmark & (\cmark) \\
      & Mish & $\rho(t) = t\tanh(\log(1+e^t))$ & - & \cmark & (\cmark) \\
      & Exponential linear unit (ELU) & $\rho(t) = \begin{cases}
        e^t - 1, & \text{if } t \leq 0 \\
        t, & \text{if } t > 0
      \end{cases}$ & - & \cmark & (\cmark) \\
      & Gaussian error linear unit (GELU) & $\rho(t) = \tfrac{t}{2}(1+\operatorname{erf}(\tfrac{t}{\sqrt{2}}))$ & - & - & \cmark \\
      & Arctan & $\rho(t) = \arctan(t)$ & - & - & \cmark \\
      \midrule
      \multirow{7}{*}{\rot{90}{Loss function}}
      & Squared error  & $\ell(y,z) = (y-z)^2$ & \cmark & (\cmark) & (\cmark) \\
      & Absolute deviation  & $\ell(y,z) = |y-z|$ & \cmark & (\cmark) & (\cmark) \\
      & Hinge  & $\ell(y,z) = \max\{0,1-yz\}$ & \cmark & (\cmark) & (\cmark) \\
      & Huber & $\ell(y,z) = \begin{cases}
        \tfrac{1}{2}(y-z)^2, & \text{if } |y-z|\leq\beta  \\
        \beta(|y-z|-\tfrac{\beta}{2}), & \text{if }|y-z|>\beta
      \end{cases}$ & \cmark & (\cmark) & (\cmark) \\
      & Logistic  & $\ell(y,z) = \log(1+e^{-yz})$ & - & \cmark & (\cmark) \\
      & Binary cross entropy & $\ell(y,z) = -(y\log z +(1-y)\log(1-z))$& - & \cmark & (\cmark) \\
      \bottomrule
    \end{tabular}
  }
\end{table}

\subsection{A prolific framework}

We argued above that many of the popular theoretical frameworks are not broad enough to account for practical functions such as ReLU-based networks.
We argue here that they are also too broad, in that they often include ill-behaved functions and sets, which don't actually appear in practice.
Crucially, this limits their ability to allow for general, wide-ranging, results.
We thus want to restrict our mathematical universe to rule out pathological objects.

\medskip\noindent
Let us first give some examples of pathological objects.
In topology, one can think about sets whose frontier have the same (or higher!) dimension, or functions that exhibit infinitely many changes in monotonicity in a bounded set; see %
\cref{nex:oscillations,nex:sethigherfrontier,nex:spirallimitcycle,nex:toposinecurve}.
In optimization theory, there are (two-dimensional) smooth convex coercive functions with curious behavior: failure of convergence of steepest descent with exact search, of central paths, or of Newton's flow, but also infinite length of Tikhonov path, or failure of the smooth Kurdyka--{\L}ojasiewicz inequality \cite{bolteCuriositiesCounterexamplesSmooth2022}.
Also, optimization over trigonometric functions is not decidable, even in the box-constrained case~\cite[3.5.5]{liberti2019undecidability}.

\medskip\noindent
In contrast with this, definability in o-minimal structures restrains the world of functions and sets to well-behaved, pathology-free objects.
As such, definable functions and sets enjoy remarkable properties.
The concepts that characterize a set being small in $\bbR$, including finiteness, countability, nowhere dense, meager, or zero Lebesgue measure become equivalent  (\cref{th:smallsets}).
There exists a definable choice lemma, which provides a constructive version of the
Axiom of Choice (\cref{definableChoice}).
A definable curve admits one-sided limits everywhere; this limit is also definable (\cref{existence_onesided_lim_omin}).
There is a unique notion of dimension of a set (\cref{omin_Dimension_Theorem}), and the frontiers of definable sets have lower dimensions than the sets (\cref{th:smallfrontier}).
These results build up to various stratification theorems: any definable set partitions into ``smooth'' definable subsets (manifolds), which fit ``nicely'' together; a parallel result holds for functions (\cref{verdier_stratification}).

\medskip\noindent
Stratification results proved particularly prolific in optimization theory.
They allowed to characterize how generalized derivatives behave on generic nonsmooth nonconvex functions \cite{bolte2007clarke}, among other results.
Eventually, this approach allowed to prove the convergence of the Stochastic Subgradient Method (SSM), also known as Stochastic Gradient Descent, on nearly any function encountered in the training of DNNs  \cite{davis2020stochastic}.
We discuss this result and its proof ideas in \Cref{sec:descentmethods}.
Focusing on tame functions and sets allowed to obtain the first theory of Automatic Differentiation (AD) on nonsmooth functions, along with a convergence proof of AD-based SSM \cite{bolte2020mathematical,bolteConservativeSetValued2021}; see \Cref{sec:remarkAD} for details.

\subsection{Content and outline}

In this document, we seek to provide a unified exposition of concepts at the interface of \emph{o-minimality}, \emph{optimization theory}, and \emph{deep learning} theory and practice.
As a motivating case-study, we expose the material that builds towards the convergence proof of the Stochastic Subgradient Method.
The ideas presented in this document are not new.
Rather, we have strived to provide proofs and arguments that are up-to-date, concise, and help develop intuition, based on material from three distinct fields, scattered across many articles.
As such, we hope that this document, along with the many articles it references, will help people to learn and use this theoretical framework.

\medskip\noindent
\Cref{sec:semialg} introduces the basic concepts and interest of definability on the familiar objects built only from (piecewise) polynomials, that is semialgebraic geometry.
\Cref{sec:defomin} discusses the nice mathematical properties of objects definable in an o-minimal structure.
\Cref{sec:descentmethods} details the key role of o-minimality in the convergence proof of SSM.

\subsection{Conventions and notation}\label{sec:conventions}
Throughout the text, $m$ and $n$ range over $\N=\{0,1,2,\ldots\}$.
$\R^>$ denotes the positive real numbers and $\varnothing$ the empty set.
If $S\subseteq A\times B$ and $a\in A$, we denote $S_{a}:=\{b\in B:(a,b)\in S\}$,
we call it a~\textbf{fiber} and we view $S$ as describing the family $(S_a)_{a\in A}$ of subsets of $B$. Usually this situation occurs where a set $R$ is given, and $A=R^m$, $B=R^n$, and $A\times B$ is identified with $R^{m+n}$.
We denote the distance between two subspaces $T,T'$ of $\R^n$ by $\Delta(T,T') \ := \ \sup_{v\in T,\|v\|=1}\operatorname{dist}(v,T')$.
The \emph{differential} of a function $h:\bbR^{\ndim}\to\bbR^{p}$ at $x\in\bbR^{\ndim}$ is the map $\D h(x):\bbR^\ndim \to \bbR^p$, defined as \(\D h(x) : \eta \mapsto \lim_{t \to 0} \frac{h(x + t\eta) - h(x)}{t}\) when the limit exists.

\medskip\noindent
We adopt notions concerning manifolds from \cite{boumal2023introduction}.
A subset $\M$ of $\bbR^{\ndim}$ is a \textbf{$\manDim$-dimensional $\C^{\smoothDeg}$-manifold} around $\crit[\vx] \in \M$
if there exists a $\C^{\smoothDeg}$ function $\maneq : \bbR^{\ndim} \to \bbR^{\ndim-\manDim}$ with a surjective derivative at $\crit[\vx] \in \M$
that satisfies for all $\vx$ close enough to $\crit[\vx]$ that $\vx \in\M$ if and only if $\maneq(\vx) = 0$.
We let $\dim(M) = d$ denote the dimension of $M$.

\medskip\noindent
Let $I\subset\bbR$ denote an interval. A~\textbf{curve} $c$ is a~continuous map $c:I\to\bbR^{\ndim}$ and its \textbf{velocity} is the derivative $\frac{d}{dt}c(t)$, also denoted $\dot{c}(t)$, wherever defined. %
A~curve $x:[0, +\infty) \to \bbR^n$ is called \textbf{absolutely continuous} if there exists a~function $y:[0, +\infty) \to \bbR^n$
that satisfies $x(t) = x(0) + \int_0^t y(\tau)d\tau$, for any $t \ge 0$.
In that case, there holds $\dot{x}(t) = y(t)$ for a.e. $t\ge 0$.

\medskip\noindent
Consider a $\manDim$-dimensional $\C^{\smoothDeg}$-manifold $\M$ included in $\bbR^n$, and a point $\vx\in\M$.
The \textbf{tangent space} at $\vx\in\M$, noted $\tangentM[\vx]$, is the set of all possible velocities of curves on $\M$ that go through $\vx$, evaluated at $\vx$.
This is a $\manDim$-dimensional linear space, which can be expressed as the null space of the differential of a manifold defining map: \(\tangentM[\vx] = \ker \D\maneq (\vx)\).
The \textbf{normal space} at point $\vx\in\M$, noted $\normalM[\vx]$, defines as the orthogonal space to $\tangentM[\vx]$ in $\bbR^{\ndim}$.
We denote the orthogonal projection onto $\tangentM[\vx]$ by $\proj_{\tangentM[\vx]}(\cdot)$.
Finally, we equip each tangent and normal space by the (restriction of the) scalar product of the ambient euclidean space $\bbR^n$, thus making the manifold Riemannian.

\section{Semialgebraic geometry}
\label{sec:semialg}

\noindent
To provide a specific, easy to visualize example, let us start with semialgebraic geometry.
We also observe that some, but not all, objects in deep learning are already definable in the semialgebraic setting.
Let us first provide an example of a semialgebraic set.

\begin{figure}[h]
    \centering
    \begin{tikzpicture}[even odd rule]
        \clip (-3,0) rectangle (4.5,2);
        \fill[c1] (-1,0) circle (1.2) (-1,0) circle (2);
        \fill[c1]  (1,0) circle (1.2)  (1,0) circle (2);
        \fill[c1] (3.5,0.5) rectangle (4,2.5);
    \end{tikzpicture}
    \caption{Illustration of the set $S$ in \cref{eq:setsemialg}.\vspace{-1ex}}
    \label{fig:semialg}
\end{figure}
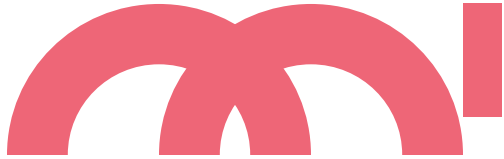

\begin{example}\label{ex:semialg}
    \Cref{fig:semialg} shows a two-dimensional semialgebraic set, defined by the following set of equations:
    \begin{equation}\label{eq:setsemialg}
    \begin{aligned}
    S &= S_1 \cup S_2, \qquad \text{where} \quad
    S_1 =
    \left\{
    (x, y):
    \begin{array}{c}
    |x-3.25| \le 0.25 \\
    |y - 1.5| \le 1
    \end{array}
    \right\},
    \\
    S_2 &=
    \left\{ (x, y) :
    \begin{array}{c}
    1.2^2 \le (x + 1)^2 + y^2 \le 2^2 \\
    y \ge 0
    \end{array}
    \right\}
    \cup
    \left\{ (x, y) :
    \begin{array}{c}
    1.2^2 \le (x - 1)^2 + y^2 \le 2^2 \\
    y \ge 0
    \end{array}
    \right\}.
    \end{aligned}
    \end{equation}
\end{example}

\subsection{Semialgebraic sets}
In \emph{semialgebraic geometry}, one studies sets defined by polynomial equalities and inequalities over real numbers, in contrast to working over complex numbers as in \emph{classical algebraic geometry}.
We begin with the notion of \emph{semialgebraic set}.

\begin{definition}\cite[Definition 2.1.4]{Bochnak2013real}\label{def:semialg}
A \textbf{semialgebraic} subset of $\R^n$ is a set of the form
\[
\bigcup_{i=1}^s\bigcap_{j=1}^{r_i}\{x\in\R^n:f_{i,j}(x)\diamond_{i,j}0\}
\]
where $f_{i,j}\in\R[X_1,\ldots,X_n]$ and $\diamond_{i,j}\in\{<,=\}$, for $i=1,\ldots,s$ and $j=1,\ldots,r_i$.
\end{definition}

\medskip\noindent
One can show that the complement of a semialgebraic set is also semialgebraic, and thus the semialgebraic sets are precisely the boolean combinations of polynomial equalities and inequalities.
\Cref{def:semialg} defines for each $n\in\bbN$ a boolean algebra $\mathcal{S}_n$ of semialgebraic subsets of $\R^n$. In the case $n=1$, we have the following complete description of the boolean algebra $\mathcal{S}_1$: any semialgebraic set in one variable is equivalent to a semialgebraic set defined using simple equalities and inequalities ``$X=a$'' and ``$X>a$'', without the need for  polynomials.
This is summarized in the following lemma.

\begin{lemma}[o-minimality property]\label{lem:o-minimal-property}\cite[Proposition 2.1.7]{Bochnak2013real}
Every semialgebraic subset of $\R$ is a finite union of intervals and points.
\end{lemma}

\noindent
In practice, the semialgebraic sets we encounter are typically subsets  of $\R^n$ for some $n\gg1$. The following key theorem enables us to study $\mathcal{S}_n$ in terms of $\mathcal{S}_m$ for some $m<n$.

\begin{tarskiSeidenberg*}\label{th:tarskiSeidenberg}\cite{Tarski48,seidenberg1954new} For $m<n$, if $A\subseteq\R^{n}$ is semialgebraic, then $\pi(A)\subseteq\R^m$ is semialgebraic, where $\pi:\R^{n}\to\R^m$ is any coordinate projection.
\end{tarskiSeidenberg*}

\noindent
For example, the following lemma can be seen as an analogue of \cref{lem:o-minimal-property} for $n>1$.

\begin{lemma}\label{lemma:semialgsetstrat}
Suppose $X\subseteq\R^n$ is semialgebraic.
Then $X$ has finitely many connected components;
moreover, there exists a finite partition $\{M_1,\ldots,M_k\}$ of $X$ such that each $M_i$ is a semialgebraic connected embedded $C^1$-manifold in $\R^n$.
\end{lemma}
\begin{proof}[Remark about proof]
There are several ways to prove this; for example, this is a consequence of Verdier stratification of sets definable in o-minimal structures;
see \cref{verdier_stratification}.
\end{proof}

\begin{example}[\cref{ex:semialg} continued.]\label{ex:semialgstratif}
  Consider again the set $S=S_1\cup S_2$ defined in \cref{eq:setsemialg}.
  The fact that $S$ has finitely many connected components is clear from the picture.
  A partition of $S$ into smooth manifolds is shown in \cref{fig:stratificationS}. We consider each connected component separately.
  The rectangle $S_1$ splits as the union of one manifold of dimension 2 (its interior), four manifolds of dimension 1 (its sides without corners), and four manifolds of dimension 0 (the four points at its corners).
  The set $S_2$ splits in one manifold of dimension 2 (its interior $I$), twelve manifolds of dimension 1 (denoted by $E_1,\ldots,E_{12}$), and twelve manifolds of dimension 0 (denoted by $P_1,\ldots,P_{12}$).
  Note that, as guaranteed by \cref{lemma:semialgsetstrat}, this partition is indeed finite. In addition, it is the one with the minimal number of elements such that each element is a connected $C^1$-manifold of $\bbR^2$.
\end{example}

\begin{figure}
\centering
\input{figures/stratificationS}
\caption{A partition of the set $S$
defined in \cref{ex:semialg} and illustrated in \cref{fig:semialg}. This partition is further discussed in \cref{ex:semialgstratif,ex:stratificationS}.}
\label{fig:stratificationS}
\end{figure}

\noindent
By \Cref{lem:o-minimal-property}, the $1$-variable semialgebraic sets $X\subseteq\R=\R^1$ satisfy a strong dichotomy:
\[
\text{$X$ is finite} \quad \text{\emph{or}} \quad \text{$X$ contains an interval.}
\]
This can be generalized to $n$-variable semialgebraic sets ($n\geq 1$) via some notion of \emph{dimension}, for example:

\begin{definition}\label{semialg_dim_def_Krull_man}
Given a semialgebraic set $A\subseteq \R^n$, define\footnote{Recall: $I(A)=\{f\in \R[X_1,\ldots,X_n]:f(a)=0\text{ for all $a\in A$}\}$ is the usual vanishing ideal of $A$. The \emph{Krull dimension} of a commutative ring is the maximal length of chains of prime ideals.}
the:
\begin{enumerate}
\item \textbf{Krull dimension} $\operatorname{dim}_{K}(A)$ of $A$  to be the Krull dimension of the ring $\R[X_1,\ldots,X_n]/I(A)$ (c.f.,~\cite[2.8.1]{Bochnak2013real}),
\item
\textbf{topological dimension} $\operatorname{dim}_t(A)$ to be the maximal integer $d\in\{0,\ldots,n\}$ such that $\pi_d[A]$ has nonempty interior in $\R^d$, where
$\pi_d : \R^n\to\R^d$
is a~projection map onto $d$ coordinates. %
\end{enumerate}
\end{definition}

\noindent
As it turns out, these dimensions are always equal (c.f.,~\cite[2.8]{Bochnak2013real}):
\[
\dim_K(A) \ = \dim_t(A),\quad\text{for every nonempty semialgebraic $A\subseteq\R^n$}.
\]
From the point of view of the upcoming \Cref{omin_Dimension_Theorem}, this agreement of dimension was inevitable the moment we observed the o-minimality of the semialgebraic setting (\cref{lem:o-minimal-property}).

\subsection{Semialgebraic functions}
Having introduced semialgebraic sets, we now turn to semialgebraic functions. Given a~function $f:A\to\R^n$ with $A\subseteq\R^m$, we say that $f$ is \textbf{semialgebraic} if its graph $\Gamma(f)\subseteq\R^{m+n}$ is a~semialgebraic set. The following two lemmas enable us
to easily recognize semialgebraic functions.

\begin{lemma}[Composability]\label{semialgebraic_composability}
The composition of semialgebraic functions is semialgebraic.
\end{lemma}
\begin{proof}
For simplicity, consider semialgebraic functions $f:\R\to\R$ and $g:\R\to\R$. By definition, there exist semialgebraic subsets $\varphi(x,y)$ and $\psi(x,y)$ of $\R^2$ such that $\Gamma(f)=\varphi(x,y)$ and $\Gamma(g)=\psi(x,y)$. By shifting variables in the definition of $g$, we may consider $\Gamma(g)$ as defined by a semialgebraic set $\psi(y,z)$. Now, if we consider both $\Gamma(f)$ and $\Gamma(g)$ as subsets of $\R^3$ involving all three variables $x,y,z$, we see by the \cref{th:tarskiSeidenberg} that:
\[
\Gamma(g\circ f) \ = \ \pi_{x,z}(\varphi(x,y)\cap \psi(y,z))
\]
is also semialgebraic. Indeed, viewing this expression in ``logical form'' it describes the set:
\[
\theta(x,z) \ := \ \exists y\ (f(x)=y \ \text{and} \ g(y)=z),
\]
which is just the definition of function composition.
\end{proof}

\begin{lemma}[Definition by case distinction]\label{lem:case_distinction}
Suppose the functions $f,g:\R^n\to\R$ and the set $X\subseteq \R^n$ are semialgebraic. Then the following function is also semialgebraic:
\[
h \ : \ \R^n\to\R,\quad x \ \mapsto \ h(x) \ := \ \begin{cases}
f(x), & \text{if $x\in X$}; \\
g(x), & \text{otherwise}.
\end{cases}
\]
\end{lemma}
\begin{proof}
The proof is left as an exercise for the reader.
\end{proof}

\begin{remark}\label{rem:table_semialgebraic}
We note that the ReLU, Softsign, Squared, Absolute deviation, Hinge and Huber functions
presented in \cref{table:definable} are semialgebraic. Indeed, they are defined
by expressions which include polynomials, their addition, multiplication and
inverse operations, and definition by case distinction. Polynomials are semialgebraic
functions by \cref{def:semialg}. Then it suffices to use \cref{semialgebraic_composability,lem:case_distinction}.
\end{remark}

\noindent
For semialgebraic functions, all one-sided limits exist in $\R\cup\{\pm\infty\}$
as is sketched in a~more general setting in \Cref{sec:tame-asymptotics}; see \cref{existence_onesided_lim_omin}. Moreover, there is a complete description of all asymptotics which can occur,
as is stated in the following lemma.

\begin{lemma}[Semialgebraic asymptotics]\cite{Miller94PowFunc}\label{lem:semialgebraic_asymptotics}
Let $f:\R\to\R$ be a semialgebraic function not eventually equal to $0$. Then there exists $c\in\R\setminus\{0\}$ and $q\in\mathbb{Q}$ such that $f(t)\sim ct^q$ as $t\to+\infty$, i.e., $\lim_{t\to+\infty}f(t)/ct^q=1$.
\end{lemma}

\noindent
It follows that many important functions in mathematics are not semialgebraic.

\begin{lemma}\label{lem:exp_not_semialgebraic}
The real exponential function $\exp:\R\to\R$ is not semialgebraic.
\end{lemma}
\begin{proof}
This follows from the observation that \cref{lem:semialgebraic_asymptotics} characterizes the available asymptotic behavior of semialgebraic functions, and the asymptotics of $\exp$ is not among them.
\end{proof}

\noindent
The need to extend the semialgebraic framework so that the exponential function
is included  led to what is known as Tarski’s problem.
\begin{tarskisProblem*}
Develop versions of results from semialgebraic geometry in a setting that also includes the real exponential function $\exp:\R\to\R$.
\end{tarskisProblem*}

\subsection{The Moment--Sum-of-Squares Hierarchy}
\label{sec:momSOS} 
It is well known that global minimization of nonconvex functions is in general extremely difficult (NP-hard).
Surprisingly, for certain optimization problems consisting exclusively of polynomials, finding global minimal values
becomes less difficult.
We take a quick detour from the main exposition on o-minimal structures, and provide a glimpse of the \emph{Moment--Sum-of-Squares method},
as one of the major success stories in the realm of \emph{semialgebraic optimization}.
For details and references, we point the reader to \cite{magronSparsePolynomialOptimization2023,Lasserre_2015}.

Consider the nonconvex optimization problem
\begin{equation}\label{eq:minpoly}
  f^\star = \inf_{x\in X} f(x), \qquad \text{ where } X = \{ x \in \bbR^n : g_1(x) \ge 0, \dots, g_m(x) \ge 0 \},
\end{equation}
and $f$, $g_1$, \dots, $g_m$ are real polynomials over $\bbR^n$.
In view of \cref{def:semialg}, both $X$ and $f$ are semialgebraic objects, thus making \cref{eq:minpoly} a semialgebraic problem.

Problem \eqref{eq:minpoly} can be cast as a linear problem over the infinite-dimensional space of probability measures:
\begin{equation}\label{eq:minmeasure}
    f_\infty = \inf_{\mu \in \mathcal{M}(X)_+} \left\{ \int_{X} f(x) d\mu(x), \quad \text{s.t.} \quad \int_X d\mu = 1 \right\},
\end{equation}
where $\mathcal{M}(X)_+$ denotes the cone of finite Borel measures supported on $X$.
To see that $f_\infty = f^\star$, denote $x^\star$ a global minimizer of \eqref{eq:minpoly} and $\delta_{x^\star}$ the corresponding Dirac measure. Then $\delta_{x^\star}$ is feasible for \eqref{eq:minmeasure}, which implies $f_\infty \le f^\star$.
Simultaneously, there holds $\int_X f(x)d\mu(x) \ge \int_X f^\star d\mu(x) = f^\star$ for any finite Borel measure of unit mass, which implies $f_\infty \ge f^\star$.
Going from formulation \eqref{eq:minpoly} to \eqref{eq:minmeasure} replaces the (difficult) nonconvexity of $f$ and $X$ by the (easier) linearity of the objective of \eqref{eq:minmeasure} and convexity of its feasible set. However, this is done at the cost of introducing a new difficulty: \eqref{eq:minmeasure} is an infinite-dimensional problem, difficult to manipulate numerically.

To address this difficulty, the moment hierarchy replaces the measure $\mu$ by finitely many of its moments $y_\alpha = \int_X x^\alpha \, d\mu(x)$, where $\alpha\in\mathbb{N}^n$.
At order $d$, one keeps the truncated moment sequence $y = (y_\alpha)_{|\alpha| \le 2d}$ and, rather than insisting that the sequence $y$ come from a genuine measure supported on $X$, which is difficult to verify, one imposes a collection of semidefinite constraints that every such measure must satisfy. This yields a tractable semidefinite optimization problem.
Replacing each monomial $x^\alpha$ in $f = \sum_\alpha f_\alpha x^\alpha$ by the corresponding variable $y_\alpha$ turns the objective into the linear form $\sum_\alpha f_\alpha y_\alpha$.
Writing $d_i = \lceil \deg(g_i)/2 \rceil$, 
for any $d \ge \max( \lceil \deg(f)/2 \rceil, d_1, \dots, d_m)$, the order-$d$ moment relaxation reads
\begin{equation}\label{eq:minmeasurek}
  f_d = \min_y \left\{ \sum_\alpha f_\alpha y_\alpha \;:\; y_0 = 1, \; M_d(y) \succeq 0, \; M_{d - d_i}(g_i \, y) \succeq 0, \; i = 1, \dots, m \right\},
\end{equation}
where $M_d(y)$ is the moment matrix, and the $M_{d - d_i}(g_i \, y)$ are the localizing matrices associated with the constraints $g_i \ge 0$.
Each value $f_d$ satisfies $f_d \le f^\star$ and, moreover, the following theoretical guarantee first obtained by Lasserre \cite{lasserreGlobalOptimizationPolynomials2001} holds: 
if one of the constraints is of the form\footnote{
    The result holds more generally under the \emph{Archimedean condition} on $X$: there exists $N>0$ such that the polynomial $N - \|x\|_2^2$ can be written as $\sigma_0 + \sum_i \sigma_i g_i$ where $\sigma_i$ are sum-of-squares polynomials.
} $g_i(x) = N - \|x\|^2$ for some $N$,
then $f_d$ is a nondecreasing sequence that converges to $f^\star$.
In addition, this convergence is \emph{generically finite}: Nie~\citet{nieOptimalityConditionsFinite2014} showed that, if every global minimizer of \eqref{eq:minpoly} meets constraint qualification, strict complementarity, and second-order sufficiency assumptions, then the relaxation is exact at some finite order: there exists a $d^\star$ such that $f_{d^\star} = f^\star$.
These conditions are met for almost every choice of polynomials $f, g_1, \dots, g_m$ of prescribed degrees, which is what makes finite convergence a \emph{generic} phenomenon.

Note that \eqref{eq:minmeasurek} admits a dual problem, which searches for the largest $\lambda$ such that $f - \lambda$ belongs to the degree-$d$ truncated quadratic module
\[
  Q_d(g) = \left\{ \sigma_0 + \sum_{i=1}^m \sigma_i g_i \;:\; \sigma_i \in \Sigma[x], \; \deg(\sigma_i g_i) \le 2d \right\},
\]
where $\Sigma[x]$ denotes the cone of sum-of-squares polynomials.
Finally, the problems \eqref{eq:minmeasurek} and their duals are convex conic (semidefinite) programs; they can thus be solved by off-the-shelf polynomial-time algorithms, such as Interior Point Methods.
Under a Slater condition, we have \emph{strong duality} \ie{} these primal and dual values coincide.

However, here as well there is no free lunch: the moment matrix $M_d(y)$ in \eqref{eq:minmeasurek} has size $\binom{n+d}{n}$, which grows like $\frac{n^d}{d!}$ as $n$ goes to infinity.
In practice, this dependence on dimension $n$ only allows one to solve \eqref{eq:minmeasurek} for $d=1$, and sometimes $d=2$, for an industrial problem.
Exploiting sparsity in the polynomial objective and constraints of problem \eqref{eq:minpoly} helps in solving larger instances \cite{magronSparsePolynomialOptimization2023}.
Finite convergence may occur for surprisingly low values of $d$, such as $d=1$ or $d=2$ for Optimal Power Flow problems \cite{wangCertifyingGlobalOptimality2022}.
Another relevant approach is recovering minimizers $x^\star$ of problem \eqref{eq:minpoly} from (approximate) solutions to problems \eqref{eq:minmeasurek}.

\section{O-minimality and tame geometry}
\label{sec:defomin}

\noindent
The arrival of \emph{o-minimality} and \emph{tame geometry} in 1984 can be attributed to three sources:
\begin{itemize}
\item First, van den Dries~\cite{vandenDries84Remarks}  introduced \emph{structures of finite type} as a framework for attacking a version of Tarski's problem. The main idea, as we will see below, is to consider an axiomatic setting which assumes an analogue of the \cref{th:tarskiSeidenberg} and in which the o-minimality property (\cref{lem:o-minimal-property}) holds.
\item  Second, Pillay and Steinhorn~\cite{PillaySteinhorn1984} recognized the analogy with \emph{strongly minimal structures}, coined the term \emph{o-minimal}, and developed much of the basic model-theoretic properties in a series of three papers ~\cite{PillaySteinhorn86I},~\cite{KnightPillaySteinhorn86II} (with Knight), and~\cite{PillaySteinhorn88III}.
\item  Later, it was realized that o-minimality is a good framework for Grothendieck's desired \emph{tame topology} from \emph{Esquisse d'un Programme}~\cite{Grothendieck84}; see~\cite{Rolin08_Establishing} for commentary.
\end{itemize}

\noindent
In this section, we will give a brief introduction to o-minimality, describe some main examples of o-minimal structures, and present some of the key properties which make an o-minimal framework conducive to \emph{convergence guarantees}.

\medskip\noindent
The standard reference for o-minimality is~\cite{driesTameTopologyOminimal1998}; other references include the introductions ~\cite{coste2000introduction,MillerFields12}, the surveys~\cite{van1999minimal},~\cite{rolin2011survey}, and the paper~\cite{van1996geometric} which serves as an introduction for geometers.

\subsection{o-minimal structures}
The definition of \emph{o-minimal structure} can be viewed as a generalization of the definition of \emph{semialgebraic set}.

\begin{definition}
    A \textbf{structure} on $\R$ is a sequence $\mathcal{R}=(\mathcal{R}_n)_{n\geq 0}$ such that for each $n\geq 0$:
    \begin{itemize}
    \item[(S1)] $\mathcal{R}_n$ is a boolean algebra of subsets of $\R^n$;
    \item[(S2)] if $A\in\mathcal{R}_n$, then $\R\times A$ and $A\times\R$ belong to $\mathcal{R}_{n+1}$;
    \item[(S3)] $\{(x_1,\ldots,x_n)\in\R^n:x_1=x_n\}\in\mathcal{R}_n$;
    \item[(S4)] if $A\in\mathcal{R}_{n+1}$, then $\pi(A)\in\mathcal{R}_{n}$, where $\pi:\R^{n+1}\to\R^n$ is the projection map on the first $n$ coordinates.
    \end{itemize}
    A structure $\mathcal{R}$ on $\R$ is \textbf{o-minimal} if it satisfies:
    \begin{itemize}
    \item[(O1)] $\{(x,y)\in\R^2:x<y\}\in\mathcal{R}_2$;
    \item[(O2)] the sets in $\mathcal{R}_1$ are exactly the finite unions of intervals and points.
    \end{itemize}
    Finally, to simplify our discussion, we shall assume all structures satisfy:
    \begin{itemize}
    \item[(S5)] $\{a\}\in\mathcal{R}_1$ for every $a\in\R$;
    \item[(S6)] $\mathcal{R}_3$ contains the graphs of addition and multiplication $+,\cdot:\R^2\to\R$.
    \end{itemize}
    \end{definition}

\noindent
Here are some comments and conventions concerning the definition:
\begin{itemize}
\item     We say a set $X\subseteq\R^n$ is \textbf{definable} (in the structure $\mathcal{R}$) if $X\in\mathcal{R}_n$.
\item We say a function $f:X\to\R^m$ is \textbf{definable} if its graph is
a~definable set.
\item We can view Axiom (S4) as imposing an analogue of Tarski-Seidenberg to hold for the definable sets.
\item Axioms (S1)-(S4) secretly say that the definable sets of a structure are precisely those sets which are definable via some \emph{first-order $\mathcal{L}$-formula $\varphi(x)$} where $\mathcal{L}$ is some \emph{first-order language}; for more on this connection to \emph{first-order logic}, see~\cite[Appendix B]{ADAMTT}.
\item Axiom (S5) says that all \emph{parameters} are automatically definable.
\item Axiom (S6) together with (S5) imposes the condition that the semialgebraic sets are definable in every o-minimal structure; one can also consider smaller o-minimal structures -- such as ``$(\R;<)$'' and ``$(\R;<,+)$'' -- although these are less immediately relevant for the purposes of training DNNs.
\item Axiom (O1) ensures that the ordering is definable; this implies that every interval is guaranteed to be definable, thanks to (S5). In fact, in the presence of (S6) the axiom (O1) is redundant in the sense that we can definably recover the ordering from the field structure via: $x<y\Leftrightarrow \exists z(z\neq 0 \ \& \ z^2=y-x)$ (we won't use this fact).
\item Axiom (O2) imposes an analogue of \cref{lem:o-minimal-property} to hold for the definable sets. Note that this is the only axiom which imposes a restriction on the definable sets.
\item By default, in this document the adjective \emph{tame} when applied to some set or function in $\R^n$ usually means \emph{definable in some o-minimal structure on $\R$}.
\end{itemize}

\medskip\noindent
The quintessential example of an o-minimal structure is the semialgebraic setting of \Cref{sec:semialg}.

\begin{example}[$\R_{\operatorname{alg}}$: the semialgebraic sets]\label{ex:semialgebraic}
For each $n\geq 0$, let $\mathcal{S}_n$ denote the collection of all semialgebraic subsets of $\R^n$; see \cref{def:semialg}. We claim that $\mathcal{S}=(\mathcal{S}_n)_{n\geq 0}$ is an o-minimal structure on $\R$ and we denote it by $\R_{\operatorname{alg}}$. By assumptions (S5) and (S6), $\R_{\operatorname{alg}}$ is the smallest o-minimal structure we will consider.
\end{example}
\noindent
We note that \cref{semialgebraic_composability} for semialgebraic functions
extends to the more general setting of functions definable in any structure.

\begin{lemma}[Composability]\label{definable_composability}
In any structure, the composition of definable functions is definable.
\end{lemma}
\begin{proof}
The same as the proof of \cref{semialgebraic_composability}, except replacing the role of the \cref{th:tarskiSeidenberg} with axiom (S4). Note that this does not require (O1), (O2), (S5) or (S6).
\end{proof}
\noindent
The o-minimality axiom (O2) prohibits ``infinite oscillation''
such as in \cref{fig:sine}.
Hence, sets and functions exhibiting this behavior cannot be definable in any o-minimal structure.

\begin{nonExample}\label{nex:oscillations}
It is clear that no o-minimal structure can define the set $\mathbb{Z}$ or the function $\sin:\R\to\R$. Moreover, one can show that the structures $(\R;<,+,\cdot,\mathbb{Z})$ and $(\R;<,+,\cdot,\sin)$ actually define the entire \emph{projective hierarchy} of subsets of $\R$ in the sense of descriptive set theory~\cite[Exercise 37.6]{Kechris95}; see also~\cite[1.2.6]{driesTameTopologyOminimal1998}.
\end{nonExample}

\begin{figure}[t]
  \centering
    \input{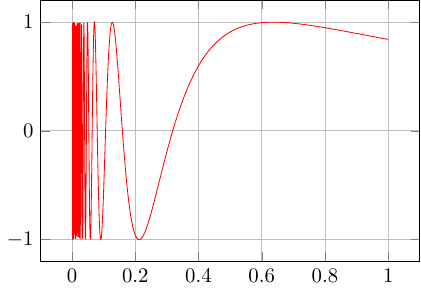}
    \caption{Topologist's sine curve $t \mapsto \sin t^{-1}$ over $(0, 1]$. This function exhibits infinite oscillations, it is not definable in any o-minimal structure; see also \cref{nex:oscillations}.
        \vspace{-1ex}}
    \label{fig:sine}
\end{figure}

\noindent
When working with subsets of $\R$, there are various notions of ``smallness'' one can
consider. In the setting of sets definable in o-minimal structures, many of these notions coincide.

\begin{smallsetstheorem*}\label{th:smallsets}
Suppose $X\subseteq\R$ is definable in an o-minimal structure. The following are equivalent:
\begin{enumerate}
\item $X$ is finite;
\item $X$ is discrete;
\item $X$ is countable;
\item $X$ is nowhere dense in $\R$;
\item $X$ has cardinality less than the continuum;
\item $X$ is meager in $\R$;
\item $X$ has Lebesgue measure zero;
\item $X$ has empty interior in $\R$.
\end{enumerate}
\end{smallsetstheorem*}
\begin{proof}
    The proof is left as an exercise for the reader.
    Definitions of the relevant topological notions can be found in \cite{Eng89}.
\end{proof}

\noindent
In many arguments in classical mathematics, one often has to appeal to the \emph{Axiom of Choice (AC)} for the existence of certain objects, e.g., a selection of a set-valued map. In the o-minimal setting, we have the following statement, where we use the fiber notation from \Cref{sec:conventions}.

\begin{definableChoice*}\label{definableChoice}\cite[Chap. 6(1.2)]{driesTameTopologyOminimal1998}
If $X\subseteq\mathbb{R}^{m+n}$ is definable in an o-minimal structure, then there exists a~definable map $\rho:\pi_m X\to\mathbb{R}^n$ such that:
\begin{enumerate}
\item $\rho(a)\in X_a$ for every $a\in\pi_mX$, and thus $\Gamma(\rho)\subseteq X$, and
\item for every $a,a'\in \mathbb{R}^m$, if $X_a=X_{a'}$, then $\rho(a)=\rho(a')$.
\end{enumerate}
\end{definableChoice*}

\subsection{Main examples}
\label{sec:examples}

In this subsection we present a few of the known o-minimal structures which are most relevant to training DNNs.
We provide a diagram of their relations in \cref{fig:ominrelation}.

\subsubsection{$\R_{\operatorname{alg}}$ (Semialgebraic sets)}
Recall that the semialgebraic sets are the smallest o-minimal structure that
we consider; see \cref{ex:semialgebraic}.

\subsubsection{$\R_{\operatorname{an}}$ (Globally subanalytic geometry)}
Here we define the structure $\R_{\operatorname{an}}$ of \emph{restricted real analytic functions}.

\begin{definition}
Let $I=[-1,1]\subset\R$.
We say $f:\R^n\to\R$ is \textbf{restricted real analytic} if:
\begin{enumerate}
\item the restriction $f|_{I^n}:I^n\to\R$ is real analytic, i.e., there exists an open neighbourhood $U\supseteq [-1,1]^n$ and an extension $\tilde{f}:U\to \R$ of $f|_{I^n}$ such that $\tilde{f}$ is real analytic on $U$, and
\item $f$ is identically zero outside $I^n$.
\end{enumerate}
Let $\R_{\operatorname{an}}$ denote the structure on $\R$ generated by all restricted real analytic functions.
\end{definition}

\begin{example}
The restricted functions $\exp:[-1,1]\to\R$ and $\sin:[0,2\pi]\to\R$, as well as the total function $\arctan:\R\to\R$ are definable in $\R_{\operatorname{an}}$; none of these are definable in $\R_{\operatorname{alg}}$.
\end{example}

\begin{theorem}[Gabrielov]
The structure $\R_{\operatorname{an}}$ is o-minimal.
\end{theorem}
\begin{proof}[Remark about proof]
This is attributed to Gabrielov who proved in 1968 a \emph{theorem of the complement} for subanalytic sets~\cite{gabrielov1968projections}.
In 1986, van den Dries noticed that o-minimality of $\R_{\operatorname{an}}$ is ``almost obvious, once observed'' given Gabrielov's result~\cite{vandenDries86Ran}.
A different treatment was given by Denef and van den Dries in~\cite{denef1988p}.
\end{proof}

\begin{lemma}\label{exp_nondef_Ran}
The exponential function $\exp:\R\to\R$ is not definable in $\R_{\operatorname{an}}$.
\end{lemma}
\begin{proof}[Remark about proof] This is true for essentially the same reason as \cref{lem:exp_not_semialgebraic}, namely, the structure $\R_{\operatorname{an}}$ is \emph{polynomially bounded} in the sense of \cref{polynomially_bounded_Ran} below.
\end{proof}

\begin{figure}
    \centering
    \resizebox{1\textwidth}{!}{
      \tikzsetnextfilename{ominstructures}
      \input{figures/o-min-structures}
    }
    \caption{Inclusion relations between the various o-minimal structures we consider. There is an arrow between two structures $\mathcal{R}_1\to\mathcal{R}_2$ to indicate $\mathcal{R}_1\subseteq\mathcal{R}_2$; see also \cref{rem:separations}. The red dashed line separates exponentially bounded structures from transexponential ones, whereas the green dashed line highlights that polynomially bounded structures form a subclass of exponentially bounded structures; see~\cref{polynomially_bounded_Ran}. The structures $\R_{\operatorname{alg}}$, $\R_{\operatorname{arctan}}$, $\R_{\operatorname{an}}$ and $\R_{\mathcal{G}}$ have the field of exponents $\Q$, whereas the remaining structures in this figure have the field of exponents $\R$.
    }%

    \label{fig:ominrelation}
\end{figure}

\subsubsection{$\R_{\exp}$ (subexponential geometry)} We now turn to the structure from Tarski's Problem:
\[
\R_{\exp} \ %
:= \  \text{the smallest structure which contains $\exp:\R\to\R$}.
\]

\begin{theorem}[Wilkie]\label{Rexp_ominimal}\cite{wilkie1996model}
The structure $\R_{\exp}$ is o-minimal.
\end{theorem}

\noindent
\Cref{Rexp_ominimal} may be regarded as an answer to a \emph{geometric version} of Tarski's Problem, however an answer to the original Tarski's Problem (decidability of $\R_{\exp}$) is an open question related to the \emph{real Schanuel's Conjecture}~\cite{MacintyreWilkie96}. The first application of \cref{Rexp_ominimal} to neural networks is in
~\cite{macintyre1993finiteness}.%

\begin{lemma}\label{cor:table_exp}
Each of the activation and loss functions in \cref{table:definable}, with the exception of $\ \operatorname{GELU}(x)$ and $\arctan(x)$, is definable in the  structure $\R_{\exp}$.
\end{lemma}
\begin{proof}[Remark about proof] The definability of all the functions (except $\operatorname{GELU}(x)$ and $\arctan(x)$) in $\R_{\exp}$ is a trivial observation that does not depend on \cref{Rexp_ominimal}; the argument proceeds exactly as in \cref{rem:table_semialgebraic} for the semialgebraic functions, using that we now have $e^x$ as an additional primitive at our disposal.
The non-definability of $\operatorname{GELU}(x)$ in $\R_{\exp}$ is nontrivial and follows from the stronger fact that $\operatorname{GELU}(x)$ is not even definable in the larger o-minimal structure $\R_{\operatorname{an},\exp}$ defined below; cf. \cref{erf_not_def_Ranexp}. The non-definability of $\arctan(x)$ follows from~\cite{bianconi1997nondefinability}, as observed in~\cite{MillerFields12}.
\end{proof}

\subsubsection{
$\R_{\operatorname{an},\exp}$ (analytic-exponential geometry)} We now consider the amalgamation:
\[
\R_{\operatorname{an},\exp} \
:= \ \text{the smallest structure which expands both $\R_{\operatorname{an}}$ and $\R_{\exp}$.}
\]

\begin{theorem}[van den Dries, Miller]
The structure $\R_{\operatorname{an},\exp}$ is o-minimal.
\end{theorem}
\begin{proof}[Remark about proof]
This was first shown in~\cite{van1994real} by adapting Wilkie's proof for $\R_{\exp}$. A~shorter proof is given in~\cite{van1994elementary} which uses \emph{valuation theory} and \emph{Hardy fields}; see \cref{rem:history}.
\end{proof}

\begin{lemma}\cite[Theorem 5.11]{DMM97LEPS}\label{erf_not_def_Ranexp}
The error function $\operatorname{erf}:\R\to\R$ is not definable in $\R_{\operatorname{an},\exp}$.
\end{lemma}
\begin{proof}[Remark about proof]
This is similar in spirit to the proof of \cref{exp_nondef_Ran}; namely, a characterization is given (using Hardy fields and transseries) of all possible asymptotics which can occur for functions definable in $\R_{\operatorname{an},\exp}$, and it is shown that the asymptotics of the error function $\operatorname{erf}$ is incompatible with these.
\end{proof}

\begin{remark}
One might wonder if in general the amalgamation of any two o-minimal structures is always o-minimal. The answer is ``no'': Le Gal~\cite{LeGal10} constructed an o-minimal structure $\R_h$ such that $\R_{\operatorname{an},h}$ is not o-minimal. In particular, there is no ``largest'' o-minimal structure. However, by Zorn's Lemma, every o-minimal structure is contained inside of a (not necessarily unique) maximal o-minimal structure.

One consequence of this, given our convention of the adjective \emph{tame}, is that the following statement is technically incorrect:
\smallskip
\begin{center}
    \emph{The composition of tame functions is tame.}
\end{center}
\smallskip
Instead, it should be understood that one actually means the following:
\smallskip
\begin{center}
    \emph{The composition of tame functions definable \\in the same o-minimal structure is tame.}
\end{center}
\smallskip
This subtlety will not affect us as the tame functions we consider in this document (the activation and loss functions) are in fact definable in a common o-minimal structure; see \cref{400_AFs_remark}.
\end{remark}

\subsubsection{
$\R_{\operatorname{Pfaff}}$ (the Pfaffian closure of $\R_{\operatorname{alg}}$)} It is natural to wonder whether an o-minimal structure is closed under solutions to differential equations, where the coefficients of said equation are definable in the given structure. In general, the answer is ``no'' since e.g. the function $\sin(x):\R\to\R$ is a~solution to the initial value problem $y''+y=0$ and $y(0)=0$.

\medskip\noindent
In contrast, the following theorem by Speissegger shows that for certain differential equations -- the so-called \emph{Pfaffian equations} -- the solutions always exist in some (possibly larger) o-minimal structure called the \emph{Pfaffian closure} $\operatorname{Pfaff}(\mathcal{R})$ of $\mathcal{R}$.
\medskip
\begin{theorem}[Pfaffian closure]~\cite{Speissegger1999Pfaff}
Let $\mathcal{R}$ be an o-minimal structure. Then there is an o-minimal expansion $\operatorname{Pfaff}(\mathcal{R})$ of $\mathcal{R}$ which is closed under solutions to Pfaffian equations in the following strong sense (where ``definable'' refers to definable in $\operatorname{Pfaff}(\mathcal{R})$). If $U$ is a~definable and connected open subset of $\R^n$, $\omega=a_1dx_1+\cdots+a_ndx_n$ is a $1$-form on $U$ with definable coefficients $a_i:U\to\R$ of class $C^1$, and $L\subseteq U$ is a Rolle leaf of $\omega=0$, then $L$ is also definable.
\end{theorem}

\noindent
For our purposes, it suffices to know the following fact about the Pfaffian closure:

\begin{corollary}\label{PfaffIntegral_fact}~\cite{Speissegger1999Pfaff}
Suppose $I\subseteq\R$ is an open interval, $a\in I$ and $g:I\to\R$ is definable in $\operatorname{Pfaff}(\mathcal{R})$ and continuous. Then its antiderivative $f:I\to\R$ given by $f(x):=\int_a^xg(t)dt$ is also definable in $\operatorname{Pfaff}(\mathcal{R})$.
\end{corollary}

\noindent
Let $\R_{\operatorname{Pfaff}}$ denote $\operatorname{Pfaff}(\R_{\operatorname{alg}})$, the Pfaffian closure of the semialgebraic structure. The following implies that $\R_{\operatorname{Pfaff}}$ contains $\R_{\exp}$, and moreover, defines the common activation and loss functions mentioned in \cref{table:definable}:

\begin{corollary}\label{cor:table_pfaff}
The following functions are definable in $\R_{\operatorname{Pfaff}}$:
\begin{enumerate}
\item the exponential function $\exp:\R\to\R$;
\item the error function $\operatorname{erf}:\R\to\R$;
\item the arctangent function $\arctan:\R\to\R$.
\end{enumerate}
In particular, all activation functions and loss functions in \cref{table:definable} are definable in $\R_{\operatorname{Pfaff}}$.
\end{corollary}
\begin{proof}
(1) The function $1/t:(0,+\infty)\to\R$ is definable in $\R_{\operatorname{alg}}\subseteq\R_{\operatorname{Pfaff}}$.
Then by \cref{PfaffIntegral_fact}, the natural logarithm $\ln:(0,+\infty)\to\R$ is also definable in $\R_{\operatorname{Pfaff}}$:
\[
\ln(x) \ = \ \int_1^x\frac{1}{t}dt.
\]
Thus the compositional inverse $\exp:\R\to\R$ is also definable in $\R_{\operatorname{Pfaff}}$; we leave the proof as an exercise, which goes along the lines of \cref{semialgebraic_composability}.

(2) It follows from (1) that the function $(2/\sqrt{\pi})\exp(-t^2):\R\to\R$ is also definable in $\R_{\operatorname{Pfaff}}$. Applying \cref{PfaffIntegral_fact} again shows that the error function $\operatorname{erf}:\R\to\R$ defined by:
\[
\operatorname{erf}(x) \ = \ \frac{2}{\sqrt{\pi}}\int_0^x\exp(-t^2)dt
\]
is also definable in $\R_{\operatorname{Pfaff}}$.

(3) The definability of $\arctan$ follows likewise using the formula:
\[
\arctan(x) \ = \ \int_0^x\frac{1}{t^2+1}dt.
\]

The final claim now follows since $\operatorname{GELU}$ and $\arctan$ are definable in $\R_{\operatorname{Pfaff}}$ by (2) and (3), and all other functions presented in \cref{table:definable} are already definable in $\R_{\exp}$; see \cref{cor:table_exp}.
\end{proof}

\begin{remark}\label{400_AFs_remark}
We now provide further heuristic justification for the claim:
\begin{quote}
\emph{Nearly all DNNs are definable in some o-minimal structure.}
\end{quote}
By \cref{definable_composability}, it suffices to consider whether the building blocks of a DNN (i.e., the activation and loss functions) are definable in some o-minimal structure; we consider the case of activation functions.
In the literature, there are various ``long lists'' of activation functions~\cite{dubey2022activation,nwankpa2018activation}.
The most exhaustive of these we could find was a recent survey~\cite{kunc2024decadesactivationscomprehensivesurvey} titled \emph{Three decades of activations: a comprehensive survey of 400 activation functions for neural networks}. In that document, we roughly estimate:\footnote{We take a few liberties with this count. For example, consider the Noisy ReLU
$f(z) = \max(0,z+a)$
where $a$ is a stochastic parameter~\cite[3.6.6]{kunc2024decadesactivationscomprehensivesurvey} . We consider this to be definable in an o-minimal structure since it is semialgebraic in the two variables $z$ and $a$. As our definition of \emph{definable in an o-minimal structure} only pertains to subsets and functions in $\R^n$, objects such as \emph{random variables} (i.e., a measurable function $\Omega\to\R$ where $\Omega$ is some abstract sample space, not a subset of $\R^n$) ostensibly fall outside the purview of \emph{definable/not-definable}.
 Likewise, consider the locally adaptive activation function (LAAF) $f(z)=g(a\cdot z)$ where $g(z)$ is a given activation function and $a$ is a trainable parameter~\cite[4.16]{kunc2024decadesactivationscomprehensivesurvey}. We also count this towards being definable in an o-minimal structure, since the process $g(z)\leadsto g(a\cdot z)$ can never break o-minimality; of course, if the initial $g(z)$ is not tame, then neither will be $f(z)$.
}
\begin{enumerate}
\item Around $89\%$ of the functions which show up are definable in one of the main examples considered above. This is by direct inspection using similar reasoning as done for the activation functions listed in \cref{table:definable}; see \cref{rem:table_semialgebraic}, \cref{cor:table_exp}, and \cref{cor:table_pfaff}.
\item Another $5\%$ of the functions involve the approach from \cite{zamora2019adaptive} of considering fractional derivatives $D^{\alpha}f(z)$ of simpler activation functions $f(z)$, where $\alpha\in\R$ is possibly a~free/trainable parameter. The definability of such a situation is a bit unclear (to us at least), although we point out that the main example is the \emph{fractional $\operatorname{ReLU}$}
defined as:
\[
\operatorname{FracReLu}_{\alpha}(z) \ = \ \frac{z^{1-\alpha}}{\Gamma(2-\alpha)}
\]
If $\alpha\in\R$ is fixed, then this is definable in $\R_{\exp}$ (even in $\R_{\operatorname{alg}}^{\R}$, the expansion of $\R_{\operatorname{alg}}$ by each fixed power function $x^{\alpha}:(0,+\infty)\to\R$ ($\alpha\in\R$), which is a polynomially bounded o-minimal structure~\cite{miller1994expansions}; see \cref{polynomially_bounded_Ran}).
Otherwise, if we want uniformity in $\alpha$, we need to work in some o-minimal structure which contains appropriate fragments of both $\exp$ and $\Gamma$. In practice, the parameter $\alpha$ is restricted so that $1<2-\alpha\leq 2$; this uses the restriction $\Gamma|_{(1,2]}$ which is definable in $\R_{\operatorname{an}}$, and thus $\R_{\operatorname{an},\exp}$ would suffice for $\operatorname{FracReLu}_{\alpha}$ under this restriction. If instead we only required $0<2-\alpha$, then $\R_{\operatorname{an},\exp}$ would no longer suffice~\cite[5.11]{DMM97LEPS}. However, van den Dries and Speissegger construct in~\cite{van2000field} a (polynomially bounded) o-minimal structure $\R_{\mathcal{G}}$ consisting of so-called \emph{Gevrey functions} such as $\phi(x):=\log\Gamma(x)-(x-\frac{1}{2})\log x+x-\frac{1}{2}\log(2\pi):(1,+\infty)\to\R$ (a function with the asymptotics of the Stirling series). Moreover, they show that adjoining $\exp:\R\to\R$ produces a structure $\R_{\mathcal{G},\exp}$ which is also o-minimal and contains $\Gamma:(0,+\infty)\to\R$; see~\cite{padgett2025definabilitycomplexfunctionsominimal} for more on the $\Gamma$ function and o-minimality.
\item The remaining $6\%$ are not definable in an o-minimal structure because they appear to use the unrestricted sine and cosine functions and thus are instances of \cref{nex:oscillations} above. Examples include:
\[
\operatorname{Cosid}(z) \ = \ \cos(z)-z\quad\text{and}\quad \operatorname{Expcos}_{a,b}(z) \ = \ \exp(-az^2)\cdot \cos(bz).
\]
A standard recourse in this situation is to reason that the functions $\sin$ and $\cos$ are actually being evaluated on some \emph{a priori fixed} standard domain $I$ such as $[0,2\pi]$, and then appeal to the fact that the restricted functions $\sin|_{I}$ and $\cos|_{I}$ are definable in $\R_{\operatorname{an}}$ whenever $I$ is a fixed compact interval. Of course we do not claim that this must always be the case.
\end{enumerate}
\noindent
Consequently, $\R_{\operatorname{Pfaff}}$ is sufficient
for nearly all functions occurring in deep-learning; at least those built from functions listed in \cref{table:definable} and discussed in (1) above. This observation has been made elsewhere, \eg{}~\cite{kranz2025sadneuralnetworksdivergent,kratsios2026feedforwardneuralnetworkdefinable}. If we wish to also consider $\operatorname{FracReLu}_{\alpha}$ as in (2) above, then one of the extensions $\operatorname{Pfaff}(\R_{\mathcal{G}})\supseteq\operatorname{Pfaff}(\R_{\operatorname{an}})\supseteq\R_{\operatorname{Pfaff}}$ would suffice, depending on the restriction of $\alpha$. The activation function $\operatorname{FracReLu}_{\alpha}$ is of theoretical interest although not widely used in practice (to our knowledge).
On the other hand, GELU is heavily used in practice:
 GELU was first defined in \cite{hendrycks2016gaussian}, first used in \cite{devlin2019bertpretrainingdeepbidirectional},
and ultimately adopted in the GPT models; see Table 1 in \cite{fang2023transformerslearnableactivationfunctions}.
\end{remark}

\noindent
We summarize the inclusion relations between the o-minimal structures discussed in this section in~\cref{fig:ominrelation}.
The following remark provides particular functions which separate various structures.

\begin{remark}[Separation between o-minimal structures]\label{rem:separations}
    The following examples illustrate some strict separations between the o-minimal structures
    in \cref{fig:ominrelation}:
    \begin{itemize}
        \item the function $\arctan$ is definable in $\R_{\arctan}:=(\R_{\operatorname{alg}},\arctan)\subseteq\R_{\operatorname{Pfaff}}$ but not in $\R_{\exp}$~\cite{bianconi1997nondefinability}; see also~\cite{MillerFields12}.
        \item the function $\exp|_{[0,1]}$ is definable in $\R_{\operatorname{an}}$ but not in $\R_{\arctan}$; see~\cite{bianconi1997nondefinability}.
        \item the function $\sin|_{[0,2\pi]}$ is definable in $\R_{\operatorname{an}}$ but not in $\R_{\exp}$; see~\cite{bianconi1997nondefinability}.
        \item the function $x^{\sqrt{2}}:(0,+\infty)\to\R$ is definable in $\R_{\operatorname{alg}}^{\R}$ but not in $\R_{\mathcal{G}}$; see~\cite{miller1994expansions,van2000field}.
        \item the function $\Gamma:(0,+\infty)\to\R$ is definable in $\R_{\mathcal{G},\exp}$ but not in $\R_{\operatorname{an},\exp}$ and the function $\phi(x):=\log\Gamma(x)-(x-\frac{1}{2})\log x+x-\frac{1}{2}\log(2\pi):(1,+\infty)\to\R$ is definable in $\R_{\mathcal{G}}$ but not in $\R_{\operatorname{an}}$; see~\cite{van2000field,DMM97LEPS}.
        Moreover, the function $\Gamma:(0,+\infty)\to\R$ is not definable in $\R_{\operatorname{Pfaff}}$ by, e.g.,~\cite[13.4]{aschenbrenner2024theorymaximalhardyfields} and~\cite{Holder1886Gamma}.
        \item the function $\operatorname{erf}(x)$ is definable in $\R_{\operatorname{Pfaff}}$ but not in $\R_{\operatorname{an},\exp}$~\cite{DMM97LEPS}.
        \item we expect the remaining expansions are also strict, although we do not have examples.
    \end{itemize}
\end{remark}

\subsection{Tame asymptotics}\label{sec:tame-asymptotics}
\emph{\jareplace{For the rest of \Cref{sec:defomin}}{In this subsection}, we fix an o-minimal structure $\mathcal{R}$, so ``definable'' means ``definable in $\mathcal{R}$''.}

\medskip\noindent
In this subsection, we shall see that an o-minimality framework gives rise to a world of ``tame asymptotics'' and thus is a setting which is conducive to ``convergence guarantees'' of various kinds. The following easy statement expresses a ``logical asymptotics'', i.e., any definable ``property'' $D$ is either eventually true or eventually false:

\begin{localOminimality*}\label{th:local-o-minimality}
Suppose $D\subseteq\mathbb{R}$ is definable \jaedit{in an o-minimal structure} and fix $a\in\mathbb{R}$. Then there exists $\varepsilon>0$ such that either:
\begin{enumerate}
\item $(a,a+\varepsilon)\subseteq D$, or
\item $(a,a+\varepsilon)\cap D=\varnothing$.
\end{enumerate}
A similar statement holds to the left of $a$, as well as at $\pm\infty$.
\end{localOminimality*}
\noindent
The power and appeal of \emph{local o-minimality} is that we can use our imagination when choosing the definable set $D$ to obtain useful consequences. For example:

\begin{lemma}[Existence of one-sided limits]\label{existence_onesided_lim_omin}
Suppose $f:(a,b)\to\mathbb{R}$ is definable \jaedit{in the o-minimal structure $\mathcal{R}$}. Then for each $c\in [a,b)$ the limit $\lim_{t\downarrow c}f(t)$ exists in $\mathbb{R}_{\pm\infty}$. Moreover, the following is also definable \jaedit{in $\mathcal{R}$}:
\[
[a,b)\to\textstyle\mathbb{R}_{\pm\infty},\quad c \ \mapsto \ \lim_{t\downarrow c}f(t).
\]
\end{lemma}
\begin{proof}[Proof sketch]
Suppose towards a contradiction there exists $\ell\in\R$ such that $\liminf_{t\downarrow c}f(t)<\ell<\limsup_{t\downarrow c}f(t)$. Then the definable set $D:=\{t\in(c,b):f(t)\geq \ell\}$ contradicts
\cref{th:local-o-minimality}.
\end{proof}
\noindent
As an application, we have the following result of Gra\~na Drummond and Peterzil~\cite{DrummondPeterzil02} which Ioffe~\cite{Ioffe08} refers to as ``the first study in which o-minimality was directly applied to optimization''.

\begin{example}[Convergence of central path]
Gra\~na Drummond and Peterzil observed in~\cite{DrummondPeterzil02} that under certain natural assumptions, if a semi-definite programming problem is definable in $\R_{\operatorname{an},\exp}$, then the central path trajectory always converges. Here we briefly recount a more general version as formulated in~\cite{Ioffe08}.
Assume that $U\subseteq\R^n$ is compact and $f:\R^n\to\R$ and $g:\R^n\times U\to\R$ are such that $f$ and each $g(\cdot,u)$ are convex continuous.
The optimization problem reads:
\[
\min_{x\in X}f(x) \quad\text{such that $g(x,u)\leq 0$, for every $u\in U$}.
\]
A common strategy is to introduce a so-called \emph{barrier function} $\beta:\R^n\to\R\cup\{+\infty\}$, where $\beta$ is finite for strictly feasible $x$ and equal to $+\infty$ otherwise, and to study the auxiliary problem with a parameter $\mu>0$:
\[
\min_{x\in \R^n}f(x)+\mu\beta(x).
\]
Let $x(\mu)$ denote a solution to the auxiliary problem (called the \emph{central path}); this might not exist or be uniquely defined.
Under this setup, the main result follows easily from \cref{existence_onesided_lim_omin}:
\begin{itemize}
\item Suppose $f,g,\beta$ are definable in an o-minimal structure. If $x(\mu)$ is well-defined and bounded for sufficiently small $\mu>0$, then the central path $x(\mu)$ converges as $\mu\downarrow 0$.
\end{itemize}
A similar result was also proved in~\cite{Halicka02} in the semialgebraic setting using the so-called \emph{curve selection lemma}, a version of \cref{definableChoice}. See~\cite{DrummondPeterzil02,Halicka02} for a further discussion of the history of this problem.
\end{example}

\noindent
We also have a set-valued version of \cref{existence_onesided_lim_omin}:

\begin{lemma}[Existence of one-sided Painlev\'{e}-Kuratowski limits]
Suppose $X\subseteq\R^{n+1}$ is definable \jaedit{in the o-minimal structure $\mathcal{R}$}, viewed as a definable family $(X_t)_{t\in\R}$ of definable sets $X_t\subseteq\R^n$ ($t\in\R$). Then the one-sided Painlev\'{e}ve-Kuratowski limit (cf.~\cite[4]{rockafellar2009variational}) of this family exists and is definable \jaedit{in $\mathcal{R}$}:
\begin{equation}\label{eq:liminfsup}
\textstyle \liminf_{t\downarrow 0}X_t \ = \ \limsup_{t\downarrow0}X_t.
\end{equation}
\end{lemma}

\noindent
As an application, two notions of \emph{tangent cone} (cf.~\cite[6A]{rockafellar2009variational}) always coincide in the tame setting:

\begin{lemma}[Tangent cone]
If $Q\subseteq\R^n$ is definable \jaedit{in an o-minimal structure}, then $Q$ is geometrically derivable, i.e., the tangent cone $T_{Q}(x)$ of $Q$ at a point $x\in Q$ equals the derivable tangent cone $T_Q^{\operatorname{der}}(x)$, i.e.,
\[
\textstyle \limsup_{t\downarrow 0}t^{-1}(Q-x) \ = \ \liminf_{t\downarrow 0}t^{-1}(Q-x).
\]
\end{lemma}

\noindent
Higher-order versions of this fact are true as well, for instance, every definable set is \emph{parabolically derivable} in the sense of~\cite[Definition 13.11]{rockafellar2009variational}. For more on coincidence of tangent cones in the tame setting see~\cite{KurdykaLeGalNguyen18}.

\medskip\noindent
Thus in general we should be on the lookout
for any properties of the form \eqref{eq:liminfsup},
as these are guaranteed to hold in the tame setting.

\medskip\noindent
We now turn our attention to the nature of 1-variable definable functions.

\begin{monotonicityTheorem*}\label{monotonicity_theorem}\cite[3.1.2]{driesTameTopologyOminimal1998}
If $f:(a,b)\to\R$ is definable \jaedit{in an o-minimal structure}, then there exist points $a=a_0<a_1<\cdots<a_k<a_{k+1}=b$ such that each restriction $f|_{(a_i,a_{i+1})}$ is continuous, and either constant or strictly monotone.
\end{monotonicityTheorem*}

\noindent
In fact, we can achieve an arbitrarily high finite degree of smoothness:

\begin{smoothMonotonicityTheorem*}\label{smoothMonotonicityTheorem*}
Given $r\geq 1$, if $f:(a,b)\to\R$ is definable \jaedit{in an o-minimal structure}, then there exist points $a=a_0<a_1<\cdots<a_k<a_{k+1}=b$ in $(a,b)$ such that each restriction $f|_{(a_i,a_{i+1})}$ is $C^r$-smooth, and either constant or strictly monotone.
\end{smoothMonotonicityTheorem*}
\noindent In order to quickly prove the \cref{smoothMonotonicityTheorem*}, we recall the classical Lebesgue's Theorem (1904) from real analysis:

\begin{quote}
\emph{If $f:(a,b)\to\R$ is monotone, then $f$ is differentiable a.e. on $(a,b)$.}
\end{quote}
\begin{proof}
Let $r=1$. By the \cref{monotonicity_theorem}, we can assume $f$ is monotone.
Since the points of (non)differentiability are definable (exercise!), and measure zero by Lebesgue's Theorem, we have from the \cref{th:smallsets} that there are only finitely many points of non-differentiability. The case $r\geq 2$ follows by induction.
\end{proof}

\noindent
Here is a more subtle property. Recall from freshman calculus that \emph{l'Hopital's rule} gives a~one-direction implication, under some \emph{indeterminate form} assumption. The reverse implication does not hold in general in the ``naive'' direction, as the following example shows:
\begin{align*}
\lim_{t\downarrow 0}\frac{t^2\sin t^{-1}}{t} \ = \ 0, \quad\text{however,} 
\quad \lim_{t\downarrow 0}\frac{(t^2\sin t^{-1})'}{t'} \ = \ \lim_{t\downarrow0}(2t\sin t^{-1}-\cos t^{-1})\quad \text{does not exist.}
\end{align*}
In the tame setting we have:

\begin{lemma}[Reverse l'Hopital's rule]\label{reverse_lhopitals_rule}
Suppose $\phi,\psi:(0,\varepsilon)\to\R$ are definable \jaedit{in an o-minimal structure} with $\lim_{t\downarrow 0}\phi(t)=\lim_{t\downarrow 0}\psi(t)=0$, and assume $\psi'(t)>0$ as $t\downarrow 0$. Then for $\ell\in\R$:
\[
\lim_{t\downarrow 0}\frac{\phi(t)}{\psi(t)} \ = \ \ell \quad \Rightarrow \quad \lim_{t\downarrow0}\frac{\phi'(t)}{\psi'(t)} \ = \ \ell.
\]
\end{lemma}
\begin{remark}[(Some) history of \cref{reverse_lhopitals_rule}]\label{rem:history}
The version here is from \emph{Tame functions are semismooth} (2009)~\cite[Lemma 1]{bolte2009tame} by Bolte, Daniilidis, and Lewis, and was the crucial observation needed to prove their main theorem that all definable locally Lipschitz functions are semismooth~\cite[Theorem 1]{bolte2009tame}; see also \cref{poly_bdd_in_optimization_remark} below.%

The history of \cref{reverse_lhopitals_rule} and its various disguised forms dates back already to the connection between o-minimality and so-called \emph{Hardy fields}, e.g.,~\cite{Miller94,van1994elementary} (1994), see also~\cite[\S3]{MillerFields12}. The Hardy field $H(\mathcal{R})$ of $\mathcal{R}$ is the ordered differential field of germs at $+\infty$ of $1$-variable definable functions $\R\to\R$. Hardy fields are important invariants of o-minimal structures as they capture all possible asymptotic behavior of 1-variable definable functions.

More recently, Aschenbrenner, van den Dries and van der Hoeven have introduced an abstract setting of so-called \emph{$H$-fields} designed to capture the first-order properties of Hardy fields.
Their work \emph{Asymptotic differential algebra and model theory of transseries} (2017)~\cite{ADAMTT} culminates in an analogue of the \cref{th:tarskiSeidenberg} in a certain natural framework for asymptotic algebraic differential equations over $H$-fields (in the same way semialgebraic geometry is a natural framework for polynomial equations over $\R$).

In the theory of $H$-fields, \cref{reverse_lhopitals_rule} shows up as e.g.,~\cite[10.1.4]{ADAMTT}. This usage in the $H$-field setting (which includes the Hardy field setting, and thus the o-minimal setting) can be traced back to at least as early as 1979 in
\emph{On the value group of a differential valuation}~\cite{rosenlicht1979value} by Rosenlicht.
\end{remark}

\noindent
As already alluded to, in theory one can characterize the asymptotics which might occur in an o-minimal structure. The simplest version of this is:

\begin{definition}\label{polynomially_bounded_Ran}\cite{MillerFields12}
We say that \jaedit{the o-minimal structure} $\mathcal{R}$ is \textbf{polynomially bounded} if for every definable $f:\R\to\R$ there exists $N\in\N$ such that eventually $|f(t)|\leq t^N$ as $t\to+\infty$. As shown in~\cite[Theorem 1.2]{MillerFields12}, this is equivalent to the seemingly stronger property: for every eventually nonzero definable $f:\R\to\R$, there exists an exponent $\gamma\in\R$ and $C\in\R$ such that:
\[
f(t) \ \sim \ Ct^{\gamma} \quad \text{as $t\to+\infty$}.
\]
We define the \textbf{field of exponents} of $\mathcal{R}$ to be the set of all $\gamma\in\R$ such that the function $t^{\gamma}:(0,+\infty)\to\R$ is definable in $\mathcal{R}$; it can be checked that this is in fact a subfield of $\R$.
\end{definition}

\noindent
The structures $\R_{\operatorname{alg}}$ and $\R_{\operatorname{an}}$ are polynomially bounded with field of exponents $\Q$ (see~\cite{vandenDries86Ran} for $\R_{\operatorname{an}}$). The structure $\R_{\operatorname{alg}}^{\R}$ mentioned in (2) of \cref{400_AFs_remark} is polynomially bounded with field of exponents $\R$. Any structure containing the exponential function, for instance $\R_{\exp}$, $\R_{\operatorname{an},\exp}$, and $\R_{\operatorname{Pfaff}}$, will not be polynomially bounded and automatically have field of exponents $\R$. We summarize this in~\cref{fig:ominrelation}. To characterize the asymptotics of o-minimal structures which are not polynomially bounded, one can use \emph{Hardy fields} and \emph{transseries} as demonstrated in~\cite{DMM97LEPS}.

\begin{remark}\label{poly_bdd_in_optimization_remark}
The property \emph{polynomially bounded} has useful practical consequences in optimization. We mention two of them here:
\begin{enumerate}
\item In \emph{Tame functions are semismooth}~\cite[Theorem 2]{bolte2009tame} the authors essentially show that in a polynomially bounded setting, the nonsmooth Newton's method converges superlinearly for definable locally Lipschitz functions with an error at the $k$th step of $O(2^{-(1+\gamma)^k})$ for some positive parameter $\gamma>0$ coming from the field of exponents.
\item In \Cref{sec:descentmethods} below, we consider \emph{qualitative} properties of the convergence of the subgradient method; we do not address the \emph{rate} of convergence.
To analyze the rate, one often uses the so-called \emph{Polyak-{\L}ojasiewicz-Kurdyka (PLK) inequalities} (also known as \emph{KL inequalities}); in this case, a polynomially bounded setting similarly can yield an exponent which can imply the superlinear convergence of a descent algorithm.
See~\cite{bentoConvergenceDescentOptimization2025} for an example of how this is done.
\end{enumerate}
\end{remark}

\noindent
Remarkably, the exponential function is unavoidable in an o-minimal structure which is not polynomially bounded:

\begin{exponentialDichotomy*}\cite{Miller94}
\jaedit{Let $\mathcal{R}$ be an arbitrary o-minimal structure}. Exactly one of the following is true:
\begin{enumerate}
\item $\mathcal{R}$ is polynomially bounded, or
\item the exponential function $\exp:\R\to\R$ is definable in $\mathcal{R}$.
\end{enumerate}
\end{exponentialDichotomy*}

\noindent
It is unknown if one can achieve growth rates faster than every iterated exponential in an o-minimal structure, i.e., every known o-minimal structure is \emph{exponentially bounded}:

\begin{openQuestion}
Does there exist an o-minimal structure which defines a \emph{transexponential} function? A transexponential function is a function $E:\R\to\R$ such that for every $n\geq 1$, eventually $E(t)\geq \exp_n(t)$ as $t\to+\infty$, where $\exp_1:=\exp$ and $\exp_{n+1}:=\exp\circ \exp_{n}$.
\end{openQuestion}

\noindent
Progress has been made on the \emph{Hardy field} fragment of this question in Padgett's thesis~\cite{padgett2022sublogarithmic}. The existence of transexponential Hardy fields was shown by Boshernitzan in (1986)~\cite{boshernitzan1986hardy}. Recently, Aschenbrenner, van den Dries, and van der Hoeven proved the following remarkable result: all maximal Hardy fields (which are necessarily transexponential~\cite{boshernitzan1986hardy}) have the same first-order theory as differential fields, and moreover assuming the Continuum Hypothesis (CH) they are all isomorphic~\cite[Corollary B]{aschenbrenner2024filling}.

\subsection{Sets of higher dimension, in o-minimal structures}

In this section, we turn our attention to properties of sets of high dimension, definable in o-minimal structures, which commonly arise in deep learning.
We discuss in turn the notions of dimension and of stratification.

\subsubsection{Dimension of sets}

The first statement partially extends the \cref{th:smallsets} into higher dimension and states that a coherent value of ``dimension'' can be assigned to every definable set.
We use the fiber notation; see \Cref{sec:conventions}.

\begin{dimensionTheorem*}\label{omin_Dimension_Theorem}
Suppose $\mathcal{R}$ is an arbitrary o-minimal structure. There exists a~unique function:
\[
\dim \ : \ \{\text{definable sets}\}\to\N\cup\{-\infty\}
\]
such that:
\begin{enumerate}[label=(D\arabic*)]
    \item\label{D1}
        \begin{enumerate}[label=(\alph*)]
            \item\label{D1a} $\dim(S)=-\infty \Leftrightarrow S=\varnothing$ for definable $S \subseteq \R^n$;
            \item\label{D1b} $\dim(\{a\})=0$ for all $a \in \R$;
            \item\label{D1c} $\dim(\R)=1$;
        \end{enumerate}
    \item\label{D2} $\dim(S_1\cup S_2)=\max\{\dim(S_1), \dim(S_2)\}$ for definable $S_1,S_2 \subseteq \R^n$;
    \item\label{D3} $\dim$ is preserved under permutation of coordinates;
    \item\label{D4} if $S \subseteq \R^{n+1}$ is definable and we set $B_i\coloneqq \{ a\in \R^n : \dim(S_a)=i \}$ for $i=0,1$, then $B_i$ is definable and
    \[
    \dim\big(\{(a,b)\in S : a\in B_i \}\big)\ =\ i + \dim(B_i), \quad\text{for $i=0,1$.}
    \]
\end{enumerate}
\end{dimensionTheorem*}

\begin{example}
If $\mathcal{R}=(\R;<,+,\cdot)$ is the structure of semialgebraic sets, then the unique dimension function $\dim$ on $\mathcal{R}$ as provided by the \cref{omin_Dimension_Theorem} coincides with the dimensions $\dim_K$ and $\dim_t$ from \cref{semialg_dim_def_Krull_man}.
\end{example}

\medskip\noindent
The following theorem from \cite[Chap. 4(1.8)]{driesTameTopologyOminimal1998} concerns the dimension of the frontier of nonempty definable sets and (as pointed out in~\cite[\S2.1]{rolin2012construction}) can be regarded as an abstract topological manifestation of the principle of ``no oscillation'' such as in \cref{nex:oscillations}.
The frontier of a~set $S$ is defined as
$\operatorname{fr} S = \operatorname{cl}S \setminus S$.

\begin{smallfrontiertheorem*}\label{th:smallfrontier}
Let $S\subseteq \mathbb{R}^m$ be a nonempty set \jaedit{definable in an o-minimal structure}. Then
\[
\dim \operatorname{fr} S \ < \ \dim S.
\]
In particular, $\dim\operatorname{cl}S=\dim S$.
\end{smallfrontiertheorem*}

\medskip\noindent
The following three examples illustrate how \cref{th:smallfrontier} can fail in the non-tame setting. The graphical illustrations
are shown in \cref{fig:sine,fig:cexsdim}.

\begin{figure}[h]
  \centering
  \begin{subfigure}{0.48\linewidth}
    \resizebox{\textwidth}{!}{
      \input{figures/spirallimit}
    }
    \caption{A one-dimensional spiral, the accumulation points of which form the unit-circle (dim. 1); see \cref{nex:spirallimitcycle}.}
  \end{subfigure}
  \hfill
  \begin{subfigure}{0.48\linewidth}
    \resizebox{\textwidth}{!}{
      \input{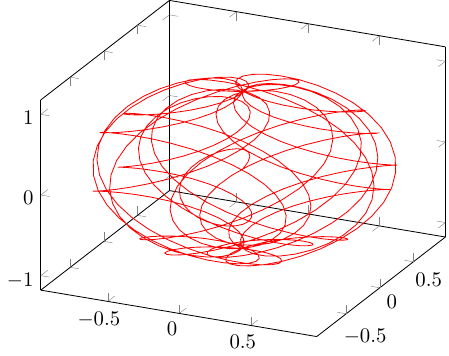}
    }
    \caption{A one-dimensional curve, the accumulation points of which form the unit-sphere (dim. 2); see \cref{nex:sethigherfrontier}.}
  \end{subfigure}
  \caption{Dimension theory: the limit of a tame curve has smaller dimension than the curve.
  The two plots show non-tame examples, where the dimension of the limit of the curve is equal \textbf{(a)} or higher \textbf{(b)} than the dimension of the curve.\label{fig:cexsdim}
    \vspace{-2ex}
  }%
\end{figure}

\begin{nonExample}[Topologist's sine curve]\label{nex:toposinecurve}
Consider the (non-tame) function $\gamma:(0,+\infty)\to\mathbb{R}$
such that $t \mapsto \sin t^{-1}$.
Then $S:=\gamma((0,+\infty))\subseteq\mathbb{R}^2$ is a $1$-dimensional $C^{1}$-manifold such that $\operatorname{fr} S=\{0\}\times [-1,1]$ also contains a nonempty $1$-dimensional $C^{1}$-manifold.
\end{nonExample}

\begin{nonExample}[Spiral with limit cycle]\label{nex:spirallimitcycle}
Consider the (non-tame) function: $\gamma:[1,+\infty)\to\mathbb{R}^2$
such that $t \mapsto (1+t^{-1})(\sin t,\cos t)$.
Then $S:=\gamma([1,+\infty))\subseteq\mathbb{R}^2$ is a $1$-dimensional $C^{1}$-manifold such that $\operatorname{fr} S$ is also a $1$-dimensional $C^{1}$-manifold. Indeed, $\operatorname{fr} S$ is the unit circle.
\end{nonExample}

\begin{nonExample}[A set whose frontier has higher dimension]\label{nex:sethigherfrontier}
Let $\alpha\in\mathbb{R}\setminus\mathbb{Q}$ be an irrational number. Consider the (non-tame) function
$\gamma:[1,+\infty)\to\mathbb{R}^3$ such that $t\mapsto (1+t^{-1})(\sin(t)\cos(\alpha t),\sin(t)\sin(\alpha t),\cos (t))$.
Then $S :=\gamma([1,+\infty))\subseteq\mathbb{R}^3$ is a $1$-dimensional $C^{1}$-manifold such that $\operatorname{fr} S$ is a $2$-dimensional $C^{1}$-manifold. Indeed, $\operatorname{fr} S$ is the unit $2$-sphere.
\end{nonExample}

\subsubsection{Stratification of sets and functions.}
Let us begin by providing some intuition.
The functions
defined by a DNN are often nonsmooth and nonconvex. However, often one may decompose the domain into finitely many regions such that the function is well-behaved on each region. See \Cref{fig:NN_intro} for an illustration of this on a two-dimensional DNN.

\begin{figure}[h]
  \begin{center} \includegraphics[width=0.3\textwidth]{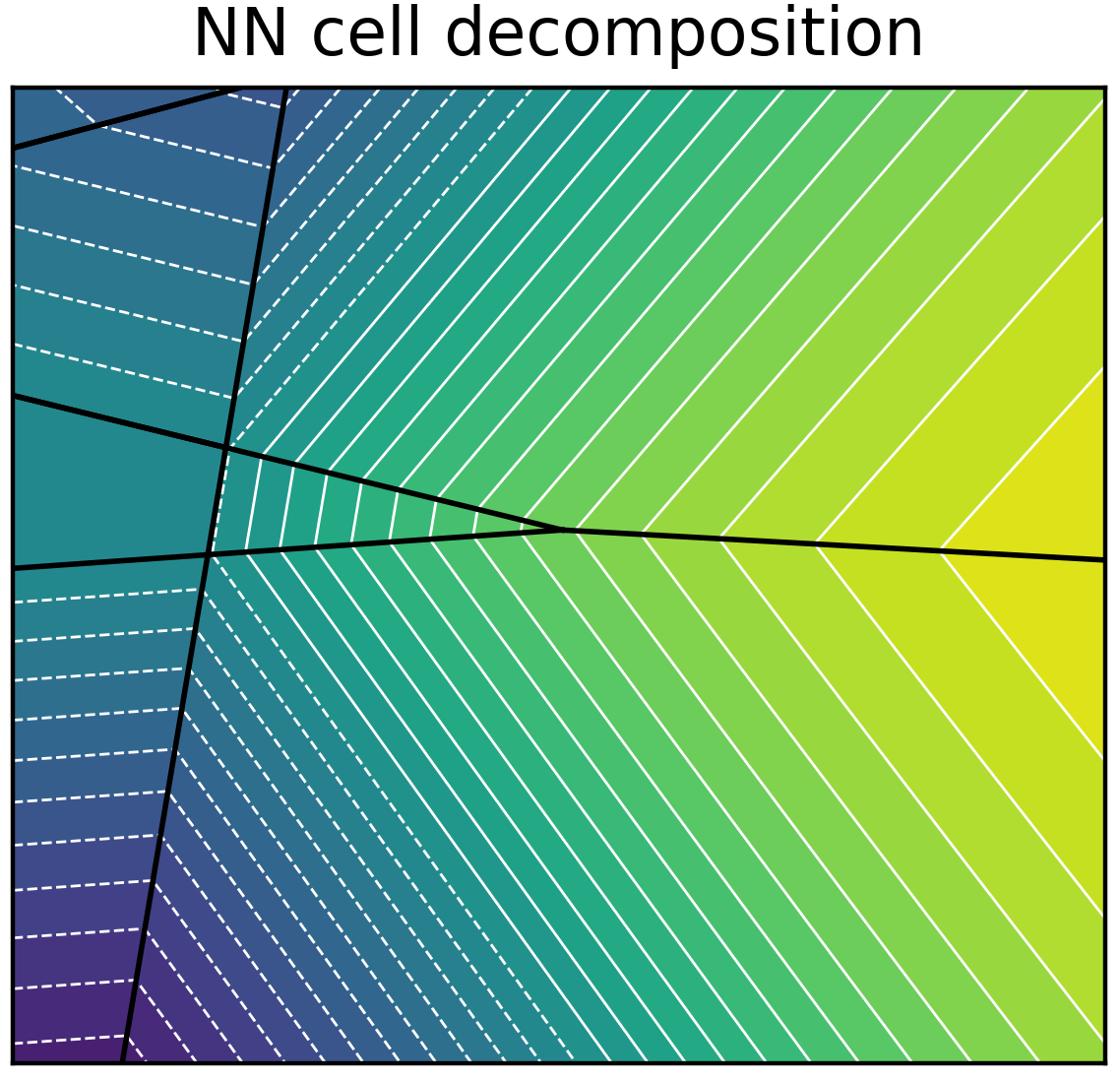}
    \caption{
      A generic 2-dimensional Deep Neural Network (MLP architecture), with level lines (white) and nonsmooth points (black) \cite{bareilles2023piecewise}.
      This is neither a convex function, nor a smooth function, although the ``cells'' indicate regions where the function is smooth.
      }
    \label{fig:NN_intro}
  \end{center}
\end{figure}

\medskip\noindent
This situation is typical for definable sets and functions.
Indeed, higher-dimensional definable sets exhibit very nice \emph{structural regularity} in the sense that they can be decomposed into finitely many definable pieces, and each piece is suitably nice; moreover, the pieces can be chosen so that adjacent pieces attach to each other in nice ways. There are several theorems of this form, for instance \emph{(Smooth) Cell Decomposition}~\cite[Chap. 3 (2.11), Chap. 7 (3.2)]{driesTameTopologyOminimal1998} and various \emph{Stratification Theorems} (e.g.~\cite{loi1996whitney,nguyen2014geometric,le1998verdier,bradley2025riso}). We mention here Loi's theorem on \emph{Verdier stratification}. First, we present some definitions, following the conventions from \Cref{sec:conventions}.

\medskip\noindent
\emph{For the rest of this section we fix an o-minimal structure $\mathcal{R}$ so ``definable'' means ``definable in $\mathcal{R}$''.}

\begin{definition}\label{def:stratifVerdier}
Given a definable set $S\subseteq\R^n$ and an integer $r\ge 1$, a~\textbf{(definable $C^r$) stratification} of $S$ is a~finite partition $\mathfrak{S}$ of $S$ into definable sets (called \textbf{strata}) such that:
\begin{itemize}
\item each stratum is a connected embedded $C^r$-manifold in $\R^n$, and
\item (frontier condition) for each $X\in\mathfrak{S}$, the relative frontier $\operatorname{cl}_S(X)\setminus X$ is a union of strata.
\end{itemize}
We say a stratification $\mathfrak{S}$ is a \textbf{Verdier} stratification if for every pair $(X,Y)$ of strata from $\mathfrak{S}$, we have that if $X\subseteq\operatorname{fr}(Y)$, then for every $\bar{x}\in X$
\begin{itemize}
\item[(w)] there exists $C>0$ and a neighbourhood $U$ of $\bar{x}$ in $\R^n$ such that:
\[
\Delta(T_X(x),T_Y(y)) \ \leq \ C\|x-y\|
\]
for every $x\in X\cap U$ and $y\in Y\cap U$.\footnote{The operator \(\Delta\) measures the distance between subspaces, see \Cref{sec:conventions}.} This is called the \textbf{Verdier condition}.
\end{itemize}
\end{definition}

\medskip\noindent
The following result provides the existence of a finite stratification for any set or function definable in some o-minimal structure.
This is a generalization of \cref{lemma:semialgsetstrat} in the semialgebraic setting, which will be instrumental for proving convergence of optimization schemes in \Cref{sec:descentmethods}.
\begin{theorem}[Loi]\label{verdier_stratification}\cite{le1998verdier}
Given $r\geq 1$, every definable set $S$ has a definable $\C^r$ Verdier stratification $\mathfrak{S}$.
Moreover, given definable sets $A_1,\ldots,A_k\subseteq S$, and a definable function $f:S\to\R$, such a stratification can be chosen so that:
\begin{enumerate}
\item $\mathfrak{S}$ partitions each $A_i$, i.e., for every stratum $M\in\mathfrak{S}$, either $M\cap A_i=\varnothing$ or $M\subseteq A_i$, and
\item for each stratum $M\in\mathfrak{S}$, the restriction $f|_{M}$ is $C^r$-smooth.
\end{enumerate}
\end{theorem}

\begin{example}[Set stratification]\label{ex:stratificationS}
    Consider again the semialgebraic set $S=S_1\cup S_2$ illustrated in \cref{fig:semialg}.
    A partition of $S$ into connected $C^1$-manifolds was proposed in \cref{ex:semialgstratif} and illustrated in \cref{fig:stratificationS}.
     This partition is in fact a stratification as the frontier condition holds for every stratum, which we demonstrate on the stratum $I$ of dimension 2, the stratum $E_1$ of dimension 1 and the stratum $P_1$ of dimension 0 (recall the notation in \cref{ex:semialgebraic}). We have
     \[
        \operatorname{cl}_S(I)\setminus I = \bigcup_{i=1}^{12} P_i\cup E_i, \quad
        \operatorname{cl}_S(E_1)\setminus E_1 = P_1\cup P_2, \quad
        \operatorname{cl}_S(P_1)\setminus P_1 = \varnothing,
     \]
     i.e., the relative interior of each strata (left-hand side) is indeed a union of strata (right-hand side).
     Moreover, this stratification is Verdier, as is every stratification in $\R^2$; see \cref{rem:r2Verdier}.
\end{example}

\begin{remark}[Stratifications in $\mathbb{R}^2$ are Verdier]\label{rem:r2Verdier}
    We comment on the fact that in $\R^2$, the condition (w) always holds, and thus every stratification in $\R^2$ is a Verdier stratification. Note that the Verdier condition applies only to pairs such that $X\subseteq \operatorname{fr}(Y)$, so in particular $\dim(X) < \dim(Y)$ by \cref{th:smallfrontier}. Consequently, either $\dim(X) = 0$ or $\dim(Y) = 2$. As soon as $X$ is a zero-dimensional manifold, or $Y$ is a full-dimensional manifold, then their tangent spaces are $\{0\}$, or $\bbR^n$, so that, for any $x\in X$ and $y \in Y$, there holds $\tangent{x}{X} \subset \tangent{y}{Y}$, and in turn $\Delta(\tangent{x}{X}, \tangent{y}{Y}) = 0$. Hence, in $\R^2$ the condition (w) holds at any $\bar{x} \in X$, for any constant $C>0$ and neighborhood $U$.
\end{remark}

\begin{figure}[t]
  \centering
  \input{figures/whitneycusp}
  \caption{The Whitney cusp $W=\{y^2+z^3-x^2z^2=0\}$, with the Verdier stratification into the smooth surface $M_2$, the punctured $x$-axis $M_1'$, and the origin $M_0$. The naive two-stratum partition $\{M_0\cup M_1',\,M_2\}$ is a stratification, but fails the Verdier condition at the origin; see \cref{nex:whitneycusp}.}
  \label{fig:whitneycusp}
\end{figure}

\begin{nonExample}[The Whitney cusp]\label{nex:whitneycusp}
  By \cref{rem:r2Verdier}, in $\R^2$ every stratification is automatically Verdier.
  Already in $\R^3$, a partition into manifolds satisfying the frontier condition may fail the Verdier condition~(w).
  A standard witness is the \emph{Whitney cusp}, the surface
  \[
    W \ = \ \{(x,y,z)\in\R^3 \ : \ y^2 + z^3 - x^2 z^2 = 0\},
  \]
  which folds along the entire $x$-axis; see \cref{fig:whitneycusp}.
  Partitioning $W$ by smoothness yields two strata: $M_1$, the $x$-axis, and $M_2 := W\setminus M_1$, the surface points admitting a $2$-dimensional Euclidean neighborhood.
  While $\{M_1,M_2\}$ is a stratification (each stratum is a connected $C^1$-manifold and the frontier condition holds), it is \emph{not} a Verdier stratification.
  To see this, take $X = M_1$ and $Y = M_2$, and consider the two sequences
  \[
    q_i \ := \ \big(\tfrac1i,\,0,\,0\big) \ \in \ M_1,
    \qquad
    p_i \ := \ \big(\tfrac1i,\,0,\,\tfrac1{i^2}\big) \ \in \ M_2.
  \]
  Denoting $e_1 = (1,0,0)$, the tangent space at $q_i$ expresses as $\tangent{q_i}{M_1} = \R e_1$.
  Besides, by denoting $F(x,y,z) = y^2 + z^3 - x^2z^2$ a manifold-defining map for $W$ (so that $W = F^{-1}(\{0\})$), the tangent space at $p_i$ expresses as $\tangent{p_i}{M_2} = \ker \D F(p_i) = \nabla F(p_i)^\perp = (-2i^{-5},\,0,\,i^{-4})^\perp$; recall \cref{sec:conventions}.
  We can now compute the distance between tangent spaces featured by the Verdier condition:
  \begin{align*}
\Delta\big(\tangent{q_i}{M_1},\,\tangent{p_i}{M_2}\big)
&= \operatorname{dist}\big(e_1,\,\tangent{p_i}{M_2}\big) \\
&= \left|\left\langle e_1,\,
\frac{\nabla F(p_i)}{\|\nabla F(p_i)\|}
\right\rangle\right| \\
&= \left|\left\langle e_1,\,
\frac{(-2i^{-1},\,0,\,1)}{\sqrt{1+4i^{-2}}}
\right\rangle\right| \\
&= \frac{2i^{-1}}{\sqrt{1+4i^{-2}}}.
\end{align*}
  Since $\|p_i - q_i\| = i^{-2}$, the ratio of the tangent space distance with the points distance goes to infinity:
  \[
    \frac{\Delta\big(\tangent{q_i}{M_1},\,\tangent{p_i}{M_2}\big)}{\|p_i-q_i\|}
    \ = \ \frac{2i}{\sqrt{1+4i^{-2}}} \ \xrightarrow[\;i\to\infty\;]{} \ +\infty .
  \]
  Thus, no constant $C$ can satisfy $\Delta(\tangent{q_i}{M_1},\tangent{p_i}{M_2}) \le C\|p_i-q_i\|$ for all large $i$, on any neighbourhood $U$ of the origin.
  Therefore, the Verdier condition~(w) fails for the pair $X = M_1$ and $Y = M_2$ at $\bar{x}=0$.

  Consider now the stratification which promotes the exceptional point to its own stratum: with $M_0 := \{0\}$, and $M_1' := M_1\setminus\{0\}$, the stratification $\{M_0, M_1', M_2\}$ \emph{is} a Verdier stratification of $W$.
  Thus, in space of dimension higher than 2, Verdier regularity is a genuine constraint that the partition must be designed to meet.
  \Cref{verdier_stratification} guarantees this is always possible.
\end{nonExample}

\begin{remark}[Verdier regularity implies Whitney]\label{remark_Verdier_implies_Whitney}
The proof of the \cref{lmm:stratfun} below requires a so-called \emph{Whitney-(a)} stratification. In fact, definable Verdier stratifications always have this property, as Loi shows that a Verdier stratification automatically has the Whitney-(b) property~\cite[1.10]{le1998verdier} in the o-minimal setting, whereas Whitney-(b) stratifications are always Whitney-(a) stratifications regardless of the setting.
See~\cite{trotman2020stratification} for more on stratifying conditions such as Verdier and Whitney-(a)/(b).
\end{remark}

\begin{remark}[Finite VC dimension]
The above discussion presented dimension and stratification results, a form of \emph{geometric tameness} for sets definable in an o-minimal structure.
For the sake of completeness, let us also mention that (high-dimensional) definable sets also enjoy \emph{combinatorial tameness} properties, that are relevant to statistical learning (although not needed for~\Cref{sec:descentmethods}).
Namely, o-minimal structures have the so-called \emph{non-independence property (NIP)}, which implies that definable families always have finite \emph{Vapnik-Chervonenkis (VC) dimension}~\cite[Chap. 5]{driesTameTopologyOminimal1998}.
This implies that definable hypothesis spaces are subject to the so-called \emph{Fundamental Theorem of Statistical Learning}, which tells us they are always \emph{PAC learnable}. For more on this part of the story see e.g.,~\cite{chase2019model,krapp2024measurability}.
\end{remark}

\section{Convergence of Subgradient Method}
\label{sec:descentmethods}

\noindent
The previous section provided convergence guarantees at the elementary level of definable curves, including the existence of one-sided limits of definable curves, and the convergence of central path methods in optimization.
In this section, we illustrate how the properties of definable objects yield convergence guarantees through the example of Stochastic Subgradient Method (SSM), also known as Stochastic Gradient Descent (SGD) (in the case where the function is differentiable).
We provide in this section an exposition of known results, notably those in Davis \emph{et al} \cite{davis2020stochastic}, and Bolte \emph{et al} \cite{bolte2007clarke}.

\medskip\noindent
In \Cref{sec:prelims}, we recall some notions of variational analysis, useful when dealing with nonsmooth objects.
In \Cref{sec:SSMcont}, we discuss a continuous time version of SSM, as a first step towards the discrete setting.
In \Cref{sec:SSMdisc}, we move to the setting of interest, where the iteration is discrete.
Finally, we provide in \Cref{sec:remarks} complementary remarks that connect the discrete setting with implemented methods.

\subsection{Preliminaries from variational analysis}
\label{sec:prelims}

We first recall the notion of Clarke subdifferential and provide some context on its relevence in nonsmooth nonconvex optimization.
Then, we complete the notion of manifold, introduced in \Cref{sec:conventions}, with  the notion of
Riemannian differential and Riemannian gradient.

\medskip\noindent
The notion of gradient, valid and useful for differentiable functions, is well extended to the nondifferentiable setting by that of Clarke subdifferential.
This notion is applicable when the points where the function is not differentiable have measure zero.
Interestingly, this holds true as soon as the function is locally Lipschitz -- this is Rademacher's theorem \cite{nekvindaSimpleProofRademacher1988}, and also by the stronger and independent stratification statement \cref{verdier_stratification}.

\begin{definition}\label{def:Clarkesubdiff}
Suppose $f:\R^n\to\R$ is a locally Lipschitz function. The \textbf{Clarke subdifferential} $\partial^cf$ is the following set-valued map:
\[
\partial^cf \ : \ \R^n\rightrightarrows\R^n,\quad x \ \mapsto \ \partial^cf(x) \ := \ \operatorname{conv}\{\lim_{i\to\infty}\nabla f(x_i):\text{$x_i\to x$, $x_i\in\operatorname{Diff}(f)$}\},
\]
where $\operatorname{Diff}(f)\subseteq\R^n$ denotes the set of points $x_i\in\R^n$ where $f$ is (Fr\'{e}chet) differentiable; see \cref{def:Frechet}. An element of the Clarke subdifferential is
called a~\textbf{subgradient}.
\end{definition}

\noindent The Clarke subdifferential captures (most of) the first-order geometry of a function.
In particular, it allows one to formulate a necessary condition on the local minimizers of a function, known as \emph{Fermat's rule}
\cite[Th. 10.1]{rockafellar2009variational}.
Except in specific cases (e.g., \Cref{sec:momSOS}),
 finding global, or even local minimizers of a general nonsmooth nonconvex function is too difficult.
In what follows, we thus settle for finding Clarke critical points.
\begin{proposition}[Fermat's rule]
    If $\bar{x}$ is a local minimizer of $f$, then $0\in\partial^c f(\bar{x})$.
\end{proposition}

\noindent We summarize below some useful properties of the Clarke subdifferential.
\begin{proposition}\label{prop:clarke}
Suppose $f:\R^n\to\R$ is a locally Lipschitz function and $x\in\bbR^n$. Then
    \begin{enumerate}[i)]
        \item if $f$ is definable in a structure, then $\partial^c f$ is also definable in that structure;
        \item $\partial^c f(x)$ is convex, compact, and nonempty;
        \item $\partial^c f(x)$ is upper semicontinuous, i.e.,  for any $\epsilon>0$, there exists $\eta>0$ such that
        \begin{equation*}
            \partial^c f(y) \subset \partial^c f(x) + \ball(0, \epsilon), \quad \text{ for any } y \in \ball(x, \eta);
        \end{equation*}
        \item if $0 \not\in \partial^c f(x)$, then there exists $\eta>0$ such that
        $0 \not\in\partial^c f(y)$ for any $y \in \ball(x, \eta)$.\label{it:uppersemi}
    \end{enumerate}
\end{proposition}

\begin{proof}
    Part i) is left as an exercise; this uses axiom (S6).
    Part ii) is proved in \cite[Proposition 2.1.2(a)]{clarkeOptimizationNonsmoothAnalysis1990}, using Theorem 2.5.1 \cite{clarkeOptimizationNonsmoothAnalysis1990} to connect \cref{def:Clarkesubdiff} to Clarke's subdifferential definition \cite[p. 27]{clarkeOptimizationNonsmoothAnalysis1990}.
    Part iii) is proved in \cite[Proposition 2.1.5(d)]{clarkeOptimizationNonsmoothAnalysis1990}.
    Part iv) is an application of part ii):
    since $0 \not\in \partial^c f(x)$ and $\partial^c f(x)$ is a~closed set, there exists $\epsilon>0$ such that $0 \not\in \partial^c f(x) + \ball(0, \epsilon)$.
    By part iii), there exists $\eta >0$ such that $\partial^c f(y) \subset \partial^c f(x) + \ball(0, \epsilon)$ for all $y \in \ball(x, \eta)$.
    Hence, $0 \not\in \partial^c f(y)$ for all $y \in \ball(x, \eta)$.
\end{proof}

\noindent

\medskip\noindent
We now recall the notion of Riemannian differential and gradient.
Consider a $\manDim$-dimensional $\C^{\smoothDeg}$-manifold $\M$ included in $\bbR^n$, a point $\vx\in\M$, and a function $\funman:\M\to\bbR$.
The \textbf{Riemannian differential} of $\funman$ at\;$x$ is the operator $\D \funman(x):\tangentM\to\bbR$ defined by $\D \funman(x)[\eta]:=\left. \frac{\mathrm{d}}{\mathrm{d}t}\funman\circ \smoothcurve(t) \right|_{t=0}$ when that limit exists, where $\smoothcurve$ is a $C^1$ curve on $\M$ such that $\smoothcurve(0)=x$ and $\smoothcurve'(0)=\eta$.
When $\D \funman(x)$ is a linear operator, we say that $f$ is differentiable relative to $M$ at $x$.
In that case, we call \textbf{Riemannian gradient} the unique vector $\grad \funman(x)$ of $\tangentM$ such that, for any tangent vector $\eta$, $\D \funman(x)[\eta] = \langle \grad \funman(x), \eta \rangle$.
Finally, we say that $f$ is \textbf{continuously differentiable}, or $C^1$, on $\M$ when
the Riemannian gradient exists and is continuous at any point of $\M$.

\medskip\noindent
We distinguish between the euclidean gradient $\nabla f(x)$, the definition of which we recall in \Cref{sec:appx_Frechet}, and
the Riemannian gradient $\grad f(x)$ relative to $\M$.
We choose not to include $\M$ in the notation $\grad f(x)$ as the manifold is always clear from context.
Moreover, if there exists a function $g:\bbR^n\to\bbR$ such that $g$ is $\C^1$ at $x$, and $f$ coincides with $g$ on $\M$ near $x$, then the Riemannian gradient of $f$ is the projection of the euclidean gradient $g$ on the tangent space:
  \(\grad f(x) = \proj_{\tangentM}(\nabla g(x))\).
One readily checks that, when $\M$ is definable, then so are $T_\M$, $N_\M$, $\proj_{T_M(\cdot)}$, $\D f$, and $\grad f$.

\subsection{Descent of subgradient curves in continuous time}
\label{sec:SSMcont}
\emph{In this subsection we fix an o-minimal structure $\mathcal{R}$, so ``definable'' means ``definable in $\mathcal{R}$''.}

\medskip\noindent
We consider a definable locally Lipschitz function $f:\bbR^n \to \bbR$ and an absolutely continuous
curve $x:[0, +\infty)\to\bbR^n$ such that
\begin{equation}\label{eq:SSMcont}
  \tag{SSM\textsuperscript{cont}}
    x'(t) \ \in \ -\partial^c f(x(t)), \quad\text{for a.e. $t\geq 0$}.
\end{equation}
Note that, while $\partial^c f$ is definable (by \cref{prop:clarke}, \emph{i)}),
it is not clear whether these conditions are sufficient for any curve $x$ that satisfies \eqref{eq:SSMcont} to be definable;
thus the following exposition does not assume $x$ to be definable.
We present the statement and proof ideas of Davis, Drusvyatskiy, Kakade and Lee \cite{davis2020stochastic}, along with earlier results of Bolte, Daniilidis, Lewis and Shiota \cite{bolte2007clarke}.
The purpose of this section is to show that $f$ decreases along trajectories of \cref{eq:SSMcont}.
Specifically, we aim to prove the following result.

\begin{proposition}[Subgradient descent]\label{lemma:subgradientdescent}
Suppose $f:\R^n\to\R$ is a definable locally Lipschitz function and $x:[0,+\infty)\to\R^n$ is an absolutely continuous curve satisfying \cref{eq:SSMcont}.
Then, for all $t \ge 0$,
\begin{equation*}
    f(x(t)) \ = \ f(x(0)) - \int_0^t\operatorname{dist}^2(0,\partial^cf(x(\tau)))d\tau.
\end{equation*}
In particular, $f\circ x$ is non-increasing.
\end{proposition}

\noindent
The following proposition is a stratification result valid for definable, locally Lipschitz, functions; it extends the basic \cref{lemma:semialgsetstrat,verdier_stratification}.
For any such function, there exists a partition of the space into $C^1$-manifolds, over which the function behaves smoothly.
In addition, any Clarke subgradient decomposes as a~sum of two perpendicular vectors: some normal vector and the Riemannian gradient.
This was first proved by Bolte \emph{et al} \cite[Proposition 4]{bolte2007clarke}.
\begin{projectionFormula*}\label{lmm:stratfun}
 Suppose $f:\R^n\to\R$ is a~definable locally Lipschitz function.
 Then, for any integer $r \ge 1$, there exists a definable $C^r$-stratification $\mathfrak{S}$ of $\mathbb{R}^n$ such that for each stratum $M\in\mathfrak{S}$, we have that $f|_M$ is $C^r$ and
\[\label{eq:projformula}
\partial^c f(x) \subset \nabla_R f(x) + \normal{x}{\M}, \quad \text{ for every } x\in M.
\]
\end{projectionFormula*}
\begin{proof}
  This result follows from \cite[Proposition 4]{bolte2007clarke}; \cite{davis2025active,bianchiConvergenceConstantStep2022} also provide similar statements.
  First, $f$ is a~definable function, thus it admits a Whitney stratification
  \cite[Lemma 8]{bolte2007clarke}; see also \cref{remark_Verdier_implies_Whitney}.
  Second, $f$ is locally Lipschitz continuous, thus it is lower semicontinuous.
  Consequently, the conclusion of \cite[Proposition 4]{bolte2007clarke} is equivalent to the claim.
\end{proof}
\noindent
When \cref{lmm:stratfun} holds, we obtain a useful geometric fact summarized in the following corollary; see \cref{fig:projformula_inclusion} for an illustration, and \cref{fig:projformula_noinclusion} for a counterexample.
\begin{corollary}\label{cor:projformula}
  Consider the setting of \cref{lmm:stratfun} and a point $x\in\M\in\stratif$.
  If there exists a vector $v\in\bbR^n$ such that $v \in \partial^c f(x)$ and $v \in \tangent{x}{\M}$,
  then,
  \begin{equation*}
    v = \grad f(x) \qquad \text{and} \qquad \|\grad f(x)\| = \dist(0, \partial^c f(x)).
  \end{equation*}
In particular, if $x$ is a Clarke critical point for $f$, then $x$ is a critical point of $f|_M$ in the usual sense, i.e., $\nabla_Rf(x)=0$.
\end{corollary}
\begin{proof}
  Let $V = \tangent{x}{\M}$, and $C = \partial^c f(x)$.
  We thus have that $V\cap C \neq \varnothing$, and $C \subset \grad f(x) + \tangent{x}{\M}^\perp$ from \cref{lmm:stratfun}. Hence \Cref{lmm:geomprojform} applies and provides the claim.
\end{proof}

\begin{figure}[t]
  \centering
  \begin{subfigure}{0.48\linewidth}
    \includegraphics[width=\textwidth]{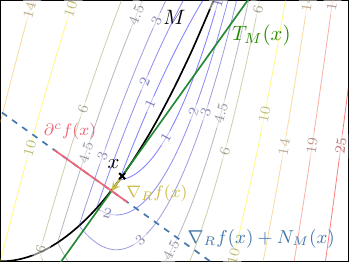}
    \caption{First case: $\nabla_R f(x) \in \partial^c f(x)$, which implies $\|\grad f(x)\| = \dist(0, \partial^c f(x))$.\label{fig:projformula_inclusion}}
    \end{subfigure}
  \begin{subfigure}{0.48\linewidth}
    \includegraphics[width=\textwidth]{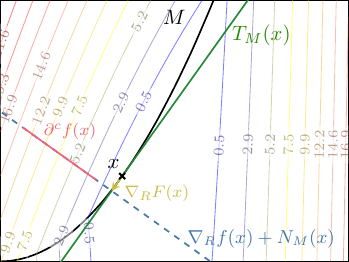}
    \caption{Second case: $\nabla_R f(x) \not\in \partial^c f(x)$, which implies $\|\grad f(x)\| < \dist(0, \partial^c f(x))$.\label{fig:projformula_noinclusion}}
  \end{subfigure}
  \caption{\label{fig:projformula}
    Illustration of \cref{lmm:stratfun} for two definable functions.
    In both cases, the stratum $\M$ is smooth, $f|_\M$ is smooth, and there holds $\partial^c f(x) \subset \nabla_R f(x) + \normal{x}{\M}$.
    In the context of a curve satisfying \cref{eq:SSMcont}, the situation \textbf{(a)} is generic, while the situation \textbf{(b)} is exceptional; see \cref{lemma:chainrule}.
  }
\end{figure}

\noindent
Using this property of the Clarke subdifferential, we can now provide an expression for the derivative of $f$ along any absolutely continuous curve $x:[0, +\infty) \to \bbR^n$.
This property has been observed many times; see \eg{} Theorem 5.8 of Davis \emph{et al} \cite{davis2020stochastic}, and the general theory from Bolte and Pauwels \cite{bolteConservativeSetValued2021}.

\begin{chainRule*}
  \label{lemma:chainrule}
  Suppose $f:\R^n\to\R$ is a~definable locally Lipschitz function and $\stratif$ is any definable $C^1$-stratification of $\R^n$ provided by \cref{lmm:stratfun}.
  Then for every absolutely continuous curve $x:[0,+\infty)\to\R^n$ we have that, for a.e. $t\ge 0$,
  \begin{enumerate}
    \item $x'(t)$ exists and $x'(t) \in \tangent{x(t)}{M_{x(t)}}$, \label{it:xdertan}
    \item $(f\circ x)'(t)$ exists and \label{it:chainrule}
      \(
        (f\circ x)'(t) \ = \  \langle \grad f(x(t)),x'(t)\rangle,
      \)
  \end{enumerate}
  where $\M_{x(t)} \in \stratif$ denotes the stratum that contains $x(t)$.
\end{chainRule*}
\noindent Before we proceed to the proof, note that we may reformulate part (\ref{it:chainrule}) as: for a.e. $t\ge 0$, $(f\circ x)'(t)$ exists and
\[
(f\circ x)'(t) \ = \ \langle v,x'(t)\rangle, \quad \text{for any } v \in \partial^c f(x(t)).
\]
This is a direct consequence of \cref{lmm:stratfun}
and part (\ref{it:xdertan}).
The derivative of $f\circ x$ thus involves the subdifferential of $f$ and the velocity of $x$ in a formula reminiscent of the smooth setting, hence the term ``chain rule''.
This is the form employed by \cite{davis2020stochastic} and others.
We prefer part (\ref{it:chainrule}) of \cref{lemma:chainrule}, as it highlights the role played by the stratification, and the importance of the smoothness of $f|_\M$.

\begin{proof}[Proof of \cref{lemma:chainrule}]
  Consider a definable locally Lipschitz function $f:\bbR^n \to \bbR$, let $\stratif$ be a definable $C^1$-stratification of $\R^n$ provided by \cref{lmm:stratfun}, and let $x:[0, +\infty) \to \bbR^n$ denote an absolutely continuous curve.
  We will prove part (\ref{it:xdertan}) with the help of a~claim using a~term ``bad''.
  We say a time $t>0$ is \emph{bad} if there exists $\varepsilon>0$ such that:
  \[
    x[(t-\varepsilon,t+\varepsilon)\setminus\{t\}]\cap M_{x(t)} \ = \ \varnothing.
  \]
    \begin{claim*}
    The set $B:=\{t\in\R^>:\text{$t$ is bad}\}$ has Lebesgue measure zero.
  \end{claim*}
    Indeed, for each $M\in \stratif$, define the set
    $B_M:=\{t\in\R^>:\text{$x(t)\in M$ and $t$ is bad}\}$.
    Since $B=\bigcup_{M\in \stratif}B_M$ and $\stratif$ has finitely many elements, it suffices to show that each $B_M$ has Lebesgue measure zero.
    By definition, each $t\in B_M$ has a neighborhood $(t-\varepsilon,t+\varepsilon)$ in which $x$ meets $M$ only at $t$.
    So each $B_M$ is discrete in $\R$, hence countable, hence of Lebesgue measure zero. This establishes the claim.

  \proofof{(\ref{it:xdertan})}
  We know that for almost every $t>0$,
  $x'(t)$ exists (from $x$ being an absolutely continuous curve)
  and $t$ is not bad (from the above claim).
  Fix such a $t$ and set $M=M_{x(t)}$.
  It suffices to prove that $x'(t)\in \tangent{x(t)}{\M}$.
  Since $t$ is not bad, there is a sequence $t_i\to t$ with $t_i\neq t$ and $x(t_i)\in M$ for each $i$ (negating the definition of ``bad'' with, for instance, $\varepsilon = i^{-1}$).
  By definition of the derivative,
  \begin{equation}\label{eq:abc}
    \lim_{i\to \infty}\frac{x(t_i)-x(t)}{t_i-t} \ = \ x'(t).
  \end{equation}
  Let $h:\bbR^n\to\bbR^{n-\dim(\M)}$ denote a $\C^1$ manifold defining map of the $\C^1$-manifold $\M$ near $x(t)$, that is, a $\C^1$ mapping such that, on a neighborhood of $x(t)$, $x\in\M$ is equivalent to $h(x)=0$, and $\D h(x)$ has rank $n-\dim(\M)$.
  Fréchet differentiability of $h$ implies, for all $i$,
  \[
    h(x(t_i)) - h(x(t)) = \D h(x(t))[x(t_i) - x(t)] + o(\| x(t_i) - x(t)\|).
  \]
  Since $x(t_i)$ and $x(t)$ belong to $\M$, the left-hand side is zero.
  Dividing by $t_i - t$, for all $i$,
  \[
    0 = \D h(x(t))\left[\frac{x(t_i) - x(t)}{t_i - t}\right] + o\left(\left\|\frac{x(t_i) - x(t)}{t_i - t}\right\|\right).
  \]
  In view of \cref{eq:abc}, taking the limit $i \to \infty$ yields
  $x'(t) \in \ker \D h(x(t))=\tangent{x(t)}{\M}$.

  \proofof{(\ref{it:chainrule})}
  Let $t>0$ be a time such that $x'(t)$ exists and $x'(t)\in \tangent{x(t)}{\M}$.
  By definition of the tangent space, there is a $C^1$ curve $\gamma:(t-\varepsilon,t+\varepsilon)\to M\subseteq\R^n$ such that $\gamma(t)=x(t)$ and $\gamma'(t)=x'(t)$; in particular, $\gamma$ is a first-order approximation of $x$ at $t$ as per \cref{def:foapp}.

  Since $f|_M$ is $C^1$, we have $f\circ \gamma=f|_M\circ\gamma$ is Fréchet differentiable at $t_{0}$ -- see \cref{def:Frechet}, which implies $f\circ x$ is Fréchet differentiable at $t_{0}$ by \cref{lmm:foapp}.
  Moreover, we have $(f\circ x)'(t)=(f|_\M \circ\gamma)'(t)$.
  Thus,
  \begin{align*}
    (f\circ x)'(t) \
    &= \ (f|_M\circ\gamma)'(t) \quad&& \\
    &= \ \D f|_M(\gamma(t))[\gamma'(t)] \quad&&\text{by def. of the Riemannian differential} \\
    &= \ \langle\grad{} f(\gamma(t)),\gamma'(t)\rangle \quad&&\text{by def. of the Riemannian gradient}\\
    &= \ \langle\grad{} f(x(t)),x'(t)\rangle \quad&&\text{by $\gamma$ being a first-order approx. of $x$ at $t$. \qedhere}
  \end{align*}
\end{proof}

\medskip \noindent
We are now ready to prove \cref{lemma:subgradientdescent}.
The main idea of this proof is that, for an absolutely continuous curve that follows the SSM dynamics \cref{eq:SSMcont}, the Riemannian gradient is (almost always) contained in the Clarke subdifferential.
Geometrically, the situation of \cref{fig:projformula_inclusion} holds generically, while the situation of \cref{fig:projformula_noinclusion} is exceptional.

\begin{proof}[Proof of \cref{lemma:subgradientdescent}]
Consider a definable locally Lipschitz function $f:\bbR^n \to \bbR$ and an absolutely continuous
curve $x:[0, +\infty)\to\bbR^n$ satisfying the continuous time SSM dynamics:
\begin{equation*}
    -x'(t) \in \partial^c f(x(t)), \quad \text{ for a.e. } t \ge 0.
\end{equation*}
Let $\stratif$ denote a definable $C^1$-stratification of $\R^n$ adapted to $f$ provided by \cref{lmm:stratfun}.
Then, \cref{lemma:chainrule} (\ref{it:xdertan}) provides
\begin{align*}
    -x'(t) \in \tangent{x(t)}{\M_{x(t)}}, \quad \text{ for a.e. } t \ge 0,
\end{align*}
where $\M_{x(t)}\in\stratif$ denotes the stratum to which $x(t)$ belongs. Applying
\Cref{cor:projformula} with $v = -x'(t)$ provides
\begin{equation}\label{eq:absec}
    x'(t) = -\nabla_R f(x(t)), \quad \text{ for a.e. } t \ge 0.
\end{equation}
Thus, there holds for a.e. $t\ge 0$,
\begin{align*}
    (f\circ x)'(t) \
    &= \ \langle \grad f(x(t)), x'(t)\rangle \quad&\text{by \cref{lemma:chainrule} (\ref{it:chainrule})}   \\
    &= \ -\|\grad f(x(t))\|^2 \quad& \text{by \cref{eq:absec}} \\
    &= \ -\operatorname{dist}^2(0,\partial^cf(x(t))) \quad &\text{ by \cref{cor:projformula}}
\end{align*}
and integrating both sides yields for any $t \ge 0$
\begin{align}
&f(x(t))-f(x(0)) \ = \ \int_0^t (f\circ x)'(\tau)d\tau \ = \ -\int_0^t\operatorname{dist}^2(0,\partial^cf(x(\tau))) d\tau. \qedhere
\end{align}
\end{proof}

\subsection{Stochastic Subgradient Method converges on tame functions: a simplified, deterministic, exposition}
\label{sec:SSMdisc}

\noindent
We note that the convergence of SSM on general definable functions is a~surprisingly recent result (2020, in Davis \emph{et al} \cite{davis2020stochastic}), and it fundamentally depends on properties guaranteed by o-minimality.
We outline the general ideas of the proof, with particular emphasis on the aspects where o-minimality plays a~critical role.
For the sake of clarity and length, we focus here on a noiseless full-batch setting, and discuss extensions to noisy and minibatch settings in the next section.

\medskip\noindent
Setting the stage, we consider finding a Clarke critical point of a given function $f:\bbR^{n}\to\bbR$, that is a point $\bar{x}$ such that $0 \in \partial^c f(\bar{x})$.
To that end, we consider using the Stochastic Subgradient Method, defined as:
\begin{equation}\label{eq:SSM_clarke}
  \tag{SSM\textsuperscript{discrete}}
    x_{k+1} = \curr + \curr[\step] \curr[y], \quad \text{ for } \curr[y] \in -\partial^c f(\curr).
\end{equation}
Our main application case is the training of DNNs, which involves, for example, the minimization problem \eqref{eq:learnpb}.
For simplicity, we consider an idealized setting where $y_{k}$ is an element of the Clarke subdifferential; this setting corresponds to full-batch subgradient method.
This simplified setting suffices to discuss the features of o-minimality.
Extensions to small-batch settings and Automatic Differentiation are discussed in \cref{rem:batch,rem:clarkeAD}.

\medskip\noindent
The convergence result is given below \cite[Corollary 5.9]{davis2020stochastic}.
\begin{proposition}\label{prop:SSM_clarke_cv}
  Suppose $f:\R^n\to\R$ is locally Lipschitz and \jaedit{definable in an o-minimal structure}, and we have a sequence $(\gamma_k)$ of step sizes in $[0,+\infty)$ and a sequence $(x_k)$ of iterates in $\R^n$ such that:
  \begin{enumerate}
    \item the sequence $(\gamma_{k})$ is non-negative, square-summable, but not summable:
      \[
      \textstyle \gamma_k \ \ge \ 0, \quad \sum_{k\geq 1}\gamma_k \ = \ +\infty,\quad \text{and}\quad \sum_{k\geq 1}\gamma_k^2 \ < \ +\infty;
      \]
    \item the sequence $(x_k)$ is bounded;
    \item the sequence $(x_k)$ satisfies the SSM iteration \cref{eq:SSM_clarke}.
  \end{enumerate}
  Then every limit point $\bar{x}$ of $(x_k)$ is Clarke critical for $f$, i.e., $0 \in \partial^{c} f(\bar{x})$, and the sequence $(f(x_k))$ converges.
\end{proposition}
\noindent
The key idea of the proof is to replace the analysis of the discrete dynamical system given by \cref{eq:SSM_clarke}, which is notoriously difficult, with that of a continuous dynamical system.
\begin{proof}[Proof sketch]
  The continuous dynamical system is given by \cref{eq:SSMcont}.

  Classical results of stochastic approximation theory guarantee that sequences $(x_{k})$ given by \cref{eq:SSM_clarke} approximate the trajectories of \cref{eq:SSMcont}, under assumptions (1) and (2) of \cref{prop:SSM_clarke_cv}; see e.g. \cite[Th. 2]{duchi2018stochastic} and \cite{benaim2006stochastic,borkar2009stochastic}.
  It thus remains to show that the trajectories of \cref{eq:SSMcont} have the desired properties, namely that they converge to points $\bar{x}$ such that $0 \in \partial^{c} f(\bar{x})$.
  Formally, this is done by proving that $f$ is a~so-called \emph{Lyapunov function} for the trajectories of \cref{eq:SSMcont}, i.e., $f$ fulfils:
  \begin{enumerate}
    \item\label{it:dense0} {\bf (Weak Sard)} For a dense set of values $r\in \bbR$, the set $\{x : f(x) = r \text{ and } 0 \in \partial^{c} f(x)\}$ is empty.
    \item\label{it:descent} {\bf (Descent)} For every
    absolutely continuous curve $x$ satisfying \cref{eq:SSMcont} such that $0 \notin \partial^{c} f(x(0))$, there exists a real $T>0$ satisfying
      \begin{equation*}%
        f(x(T))<\sup_{t\in [0,T]} f(x(t)) \leq f(x(0)).
      \end{equation*}
  \end{enumerate}

  O-minimality is key to proving the Weak Sard and Descent properties for functions $f$ matching realistic settings, such as training of DNNs.
  Indeed, o-minimality provides an easy-to-check criterion, formalized by the composability
  (\cref{definable_composability}) and the guarantee that $f$ admits a~stratification (\cref{lmm:stratfun}).

  We first consider the Weak Sard property, and begin with a general observation.
  Suppose $r\geq 1$ is arbitrary and $\frak{S}$ is an arbitrary definable $C^r$-stratification as provided by the \cref{lmm:stratfun}.
  Then \cref{cor:projformula} implies that the set of Clarke critical points of $f$ are a subset of the union of all critical points of the $f|_M$, for $M\in\frak{S}$.
  Therefore, the Clarke critical values will be a subset of the union of all critical values of $f|_M$, as $M$ ranges over $\frak{S}$.
  
Next, we want to apply the Classical Sard's Theorem~\cite{sardMeasureCriticalValues1942}.
In order to do this, we choose $r\geq 1$ sufficiently large ($r\geq n$ will suffice), and choose a definable $C^r$-stratification $\frak{S}$ as above. By our choice of $r$, the Classical Sard's Theorem tells us for each $M\in\frak{S}$ that the set of critical values of $f|_M$ is ``small'' (Lebesgue measure zero). Since $f|_M$ is definable and $C^r$, the set of its critical points is definable; hence its image, the set of critical values, is definable. Thus by the \cref{th:smallsets} it follows that $f|_M$ has only finitely many critical values. By the above observation, it follows that the set of Clarke critical values of $f$ is also finite, hence its complement is dense, as required by Weak Sard.
We refer to \cite[Cor. 9(ii)]{bolte2007clarke}, and \cite[Lem. 5.7]{davis2020stochastic} for other proofs.

  To prove the Descent property, recall that \cref{lemma:subgradientdescent} shows that a~trajectory $x$ of \cref{eq:SSMcont} satisfies, for any $t \ge 0$,
  \begin{equation}\label{eq:nsintegral}
        f(x(t)) = f(x(0)) - \int_0^t \dist^2(0, \partial^c f(x(\tau))) d\tau
  \end{equation}
    and, in particular, $f\circ x$ is non-increasing.
    Assume in addition that $0 \not\in -\partial^c f(x(0))$.
    Then by \cref{prop:clarke} \ref{it:uppersemi},
    there exists $T>0$ and $\eta>0$ such that $\operatorname{dist}^2(0,\partial^cf(x(t)))\geq\eta$ for all $t\in [0,T]$.
    Therefore, $\int_0^t \dist^2(0, \partial^c f(x(\tau))d\tau$ increases strictly for $t \in [0, T]$.
    In view of \cref{eq:nsintegral}, $f\circ x$ is strictly decreasing over $[0, T]$.
    Hence, the Descent property holds.

  Consequently, $f$ is a continuous, lower-bounded function that meets the Weak Sard and Descent properties.
  Thus, Theorem 3.2 from Davis \emph{et al} \cite{davis2020stochastic} applies and provides that any limit point $\bar{x}$ of $(x_k)$ is such that $0 \in \partial^c f(\bar{x})$ and the sequence $(f(x_k))$ converges.
\end{proof}

\subsection{The broader picture around Stochastic Subgradient Method}
\label{sec:remarks}

We presented above a convergence proof for a simplified version of Stochastic Subgradient Method, described in \cref{eq:SSM_clarke}.
Here, we discuss how the gap between this simplified setting and practice can be bridged.

\subsubsection{On the need to account for nonsmoothness}
\label{sec:needtoaccountforns}
  Rademacher's theorem states that, if a~function is Lipschitz continuous, then the set of points where it is nonsmooth has measure zero.
  Therefore optimization algorithms are unlikely to encounter such points--especially so when the iteration features noise, as is the case of \cref{eq:SSM_clarke}.
  One might then be led to give credit to the following statement, or a variation of it.
  \begin{quote}
  \emph{
      Since with probability one, the function is differentiable at iterates, everything happens as if the function was differentiable, and thus smooth convergence arguments apply directly.
      }
  \end{quote}
  Unfortunately, things are not so simple.
  Indeed, methods for smooth optimization have vastly different behaviors on nonsmooth functions, as we now illustrate.
  \Cref{fig:labela} displays the behavior of SSM and two other optimization methods on a particular nonsmooth function.
  Overall, the behavior of the methods is driven by the nonsmoothness of the function (its stratification), and particularly so by $\M^\star$, the stratum that contains the minimizer.
  SSM incurs strong oscillations perpendicular to $\M^\star$.
  The BFGS method, a~state-of-the-art method for minimizing $C^{1}$ functions with superlinear convergence speed, incurs oscillations and ultimately converges linearly at best on nonsmooth functions \cite{lewisNonsmoothOptimizationQuasiNewton2013}.
  The proximal gradient method can be applied only to a restricted class of nonsmooth functions, including that of \cref{fig:labela} \cite{beckFirstOrderMethodsOptimization2017}.
  Contrary to SSM and BFGS, it generates iterates that are attracted to nonsmooth points -- the cartesian axes here.
  Even more, the iterates belong to $\M^\star$ after a finite number of iterations; this can be used to accelerate the proximal gradient method \cite{bareillesNewtonAccelerationManifolds2022}.
  This illustrates the need to account for nonsmoothness when analyzing optimization methods.

\begin{figure}[t]
  \centering
  \resizebox{0.48\textwidth}{!}{
    \tikzsetnextfilename{Lasso_pgnsbfgs}
    \input{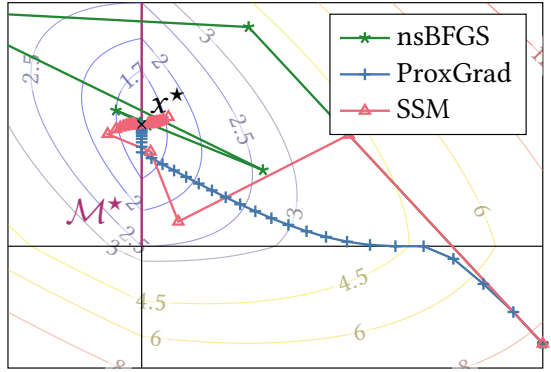}
  }
  \caption{\label{fig:labela}Illustration of \Cref{sec:needtoaccountforns}. The behavior of optimization algorithms is driven by the optimal stratum $\opt[\M] = \{ (0, u)^\top \;:\; u>0 \}$. Here, we show the iterates of noiseless full-batch SSM \eqref{eq:SSM_clarke}, the nonsmooth-BFGS \citep{lewisNonsmoothOptimizationQuasiNewton2013} and the Proximal Gradient \cite[Chap. 10.4]{beckFirstOrderMethodsOptimization2017} algorithms, on a~LASSO-type objective function $F(x) = \|Ax-b\|^{2}_{2} + \lambda \|x\|_{1}$. }
\end{figure}

\subsubsection{From Clarke critical points to local minimizers}
\label{rem:clarkeAD}
  \Cref{prop:SSM_clarke_cv} shows that the iterates of SSM converge to the Clarke critical points.
  Such points include local minima, but also saddle points which may trap iterates for a number of iterations.
  Davis \emph{et al.} \cite{davis2025active} prove that SSM avoids, or more precisely, eventually escapes from neighborhoods of, saddle points.
  Stratifications are once more instrumental to this result.

\subsubsection{From full-batch to small-batch iteration}
\label{rem:batch}
  The above exposition focused on SSM in the simple setting of \cref{eq:SSM_clarke}, which assumes that a subgradient of $f$ is computed at each iteration.
  Using stochastic approximation theory \cite{duchi2018stochastic,benaim2006stochastic,borkar2009stochastic}, this setting extends to
  \begin{equation}\label{eq:SSM_smallbatch}
    x_{k+1} = \curr + \curr[\step] (\curr[y] + \curr[\xi]), \quad \text{ for } \curr[y] \in - \partial^c f(\curr),
  \end{equation}
  where $\xi_{k}$ is a random variable with zero mean, and typically bounded variance.
  In the context of training DNNs, where $f$ is written  as a large sum (recall \eqref{eq:learnpb}), the extension \eqref{eq:SSM_smallbatch} accounts for the more realistic version of \eqref{eq:SSM_clarke} where $y_{k}$ is computed relative to one or a few elements of the sum, as opposed to all elements of the sum.
  This is the small-batch setting, as opposed to the full-batch setting of \eqref{eq:SSM_clarke}.

\begin{figure}[t]
  \begin{center}
    \includegraphics[width=0.5\textwidth]{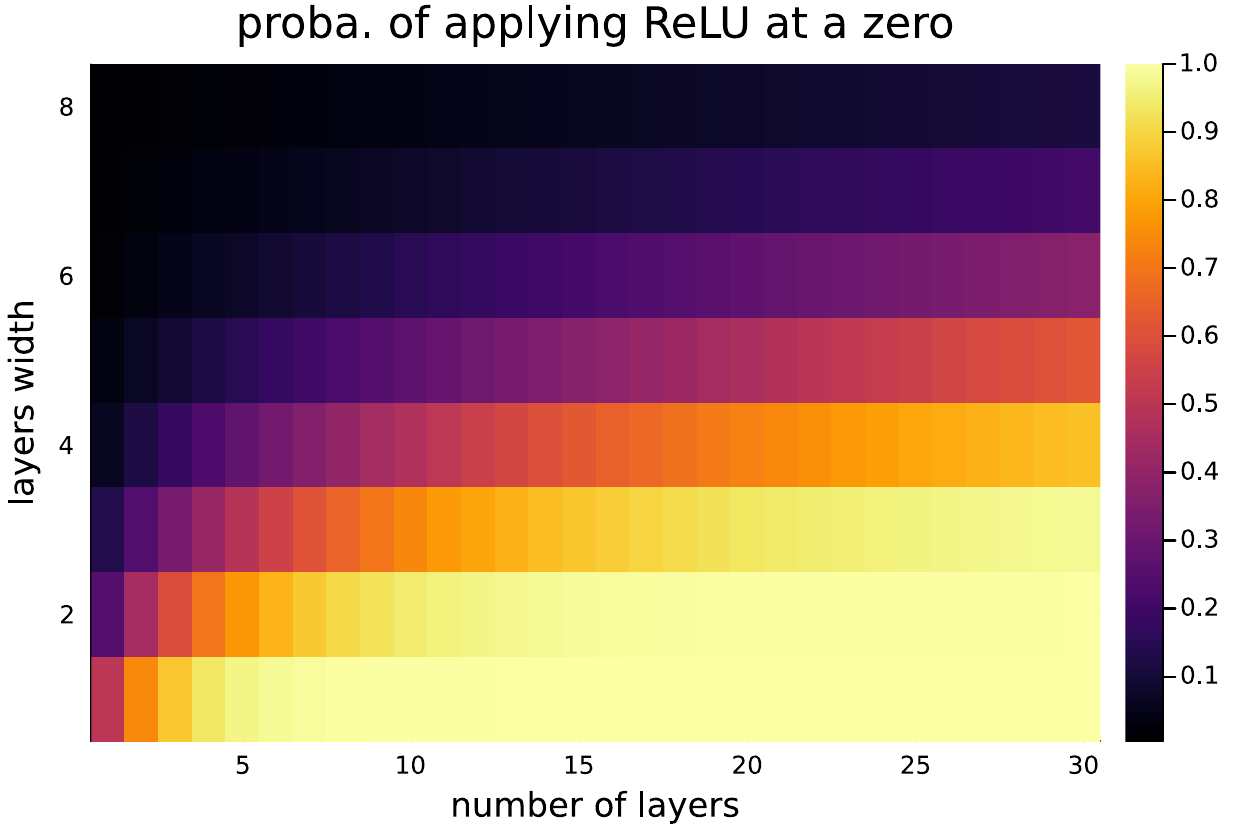}
    \caption{
      Chances of evaluating an activation function at a nonsmooth point for feedforward ReLU Neural Networks of varying number of layers and layer width, for random weights and evaluation points. Probabilities are computed with $10^{5}$ samples.
    }
    \label{fig:ReLUactivity}
  \end{center}
\end{figure}

\subsubsection{Automatic Differentiation of definable mappings}
\label{sec:remarkAD}

\noindent
In order to implement the SSM iteration \cref{eq:SSM_clarke}, one must compute an element $y_k$ of the (negation of the) Clarke subdifferential of $f$ at $x_k$.
Computing such a ``derivative'' is a challenging task, especially so in the context of training DNNs, where the function to optimize involves a deep neural network, which itself is a large composition of elementary functions; see \cref{eq:defDNN,eq:learnpb}.
We now discuss some aspects of this challenging task.

Automatic Differentiation (AD) groups methods that describe how to compute the derivative of a function that is written as a composition of elementary functions with known derivatives.
To fix ideas, consider a function
\begin{equation}
  \label{eq:compo}
  f = g_1 \circ \dots \circ g_m,
\end{equation}
where $g_1$, \dots, $g_m$ are elementary functions.
Note that this covers the training of DNNs, as sketched in \cref{eq:defDNN,eq:learnpb} in the introduction.
AD methods essentially implement the chain rule: assuming that $g_1$ and $g_2$ are two differentiable mappings, their composition is also differentiable, and their Jacobians\footnote{The Jacobian is a natural generalization of the derivative to multivariate functions; see \cref{def:Frechet}.} satisfy $J_{g_1\circ g_2}(x) = J_{g_1}(g_2(x)) J_{g_2}(x)$.
This extends directly to the composition of $m$ functions.
Thus, AD methods are guaranteed to exactly compute the Jacobian $J_f(x)$ as soon as all the elementary functions $g_i$ are differentiable; see \eg{} \cite{griewankEvaluatingDerivativesPrinciples2008} and references therein.

At this stage, one may be led to give credit to the following statement:
\begin{quote}
  \emph{Given any differentiable function $f$, say with a compositional expression as per \cref{eq:compo}, an AD method provides the correct derivative at any point.}
\end{quote}
Again, things are more complex, as illustrated by the following examples. We denote by $x^+=\max(0,x)$ the ReLU operator.
\begin{itemize}
  \item Consider the \emph{smooth} function defined by $f(x) = ((-x)^+ + x) - x^+$, and an AD algorithm which provides the correct derivative of $x \mapsto x^+$ when $x\neq 0$, and the value $0$ when $x = 0$. A reasonable choice, since $\partial^c (\cdot)^+ (0) = [0, 1]$.
  Then, $f$ is the null function, and yet the AD algorithm specified above would provide derivative $1$ at $x = 0$; see \cite{bolte2020mathematical} for details.
  Thus, AD on a smooth function may fail to compute the correct derivative.
  \item When $f$ is \emph{nonsmooth}, the situation is similar. Consider $f(x) = x^+ - \tfrac{1}{2}(-x)^+$. This function is nonsmooth at $0$, with Clarke subdifferential $\partial^c f(0) = [\tfrac{1}{2}, 1]$.
  The same AD algorithm as above would provide derivative $0$ at $x = 0$.
  \item To complete the picture, we reproduce in \cref{fig:ReLUactivity} the experiment of \cite{bolteConservativeSetValued2021} illustrating that, given a random point and a feedforward ReLU-based neural network with random weights, there is a positive probability that some ReLU be evaluated at a nonsmooth point.
\end{itemize}
We refer to \cite{bolte2020mathematical,kakadeProvablyCorrectAutomatic2018,lee2020correctness} for other problematic examples.
The bottom line is this: as soon as a function $f$ contains a nonsmooth component $g_i$, the output of an AD method need not correspond to the usual derivative, be it the Euclidean gradient when $f$ is smooth, or an element of the Clarke subdifferential when $f$ is nonsmooth.

At this stage, it should be clear that the behavior of AD on functions that feature nonsmooth components is surprisingly not straightforward, and deserves investigation.
To address this issue, Bolte and Pauwels introduced in 2020 a theory that formalizes AD methods as instances of \emph{Conservative Fields} \cite{bolte2020mathematical,bolteConservativeSetValued2021}.
One key takeaway is that any definable Conservative Field applied to a definable function returns the gradient of the function almost everywhere; see Th.~8, Coro.~5 and Rem.~11 in \cite{bolte2020mathematical}.
Besides, any convex-valued conservative field contains the Clarke subdifferential.
Another notable feature is that \cref{cor:projformula} for Clarke's subdifferential extends to conservative fields: Lewis and Tian \cite[Th. 2.2]{lewisStructureConservativeGradient2021} show that, for a conservative field of a function, there exists a Whitney stratification of that function (\cref{def:stratifVerdier}) such that, at any point, the conservative field is contained in the sum of the Clarke subdifferential of the function with the normal space to the stratum at that point.
Finally, this theory allows one to give convergence guarantees for AD-based implementations of SSM, such as in PyTorch \cite{paszke2017automatic} or TensorFlow \cite{abadi2016tensorflow}, thus bridging the gap between the convergence result detailed above and actual implementations; see \eg{} \cite[Th.~4]{bolte2020mathematical}.

In conclusion, the definability of the functions provides a notion of \emph{chain rule} for a variety of notions of derivatives.
Such notions include the Clarke subdifferential, recall \cref{lemma:chainrule}, and AD methods, see \eg{} \cite[Th. 3]{bolte2020mathematical} and \cite[Coro. 5]{bolteConservativeSetValued2021}.
This observation provides a theoretical understanding of AD, which is a solid theoretical foundation to develop convergence guarantees for AD-based methods, under a realistic set of assumptions on the functions and algorithms.

\section{Conclusion}

\noindent
In this expository document, we have discussed ways in which tame geometry provides a good framework for the study of optimization and deep learning: it is both realistic, in that it encompasses current deep-learning architectures, and prolific, in that it allows one to develop nontrivial theoretical guarantees that are relevant in~practice.

\medskip\noindent
The realism of o-minimality is shown by the fact that a vast majority of DNNs are definable in an o-minimal structure.
This follows notably from the remarkable composability guarantees in tame geometry.
In stark contrast, all modern theoretical frameworks do not account for DNNs, even simple ReLU-based feedforward MultiLayer Percetprons.
Such frameworks include: convexity, weak convexity, smoothness (e.g., continuous differentiability with Lipschitz first derivative), and subdifferential regularity.
Yet, the activation and loss functions encountered in the training of Deep Neural Networks (DNNs) (see \cref{table:definable} and \cref{400_AFs_remark}) feature nonconvexity and nonsmoothness.
Specifically, DNNs are in general not differentiable, so the smoothness assumption does not apply (gradients are discontinuous; think of absolute value), and feature nonsmoothness that cannot be accounted for by a convexity assumption, or any weakened form (e.g., weak convexity, prox-regularity, or even Clarke regularity; think of minus absolute value).

\medskip\noindent
The prolific nature, while not yet fully realized, is also becoming apparent.
While optimization over trigonometric functions, for instance, is not decidable even in the box-constrained case, see \eg{} \cite[3.5.5]{liberti2019undecidability}, optimization in tame geometry is computable, at least in the sense of first-order critical points, both in theory and practice.
The chain rule on some generalized derivatives enables automatic differentiation in PyTorch or TensorFlow. It is consequently a key component for Deep Learning practice, while at the same time intimately tied to tame geometry.
We hope that the prolific nature of o-minimality will motivate future research.

\appendix
\section{Some facts used in the proofs of \Cref{sec:descentmethods}}

\subsection{A geometric property of inner product spaces}

We provide a quick geometric fact about projection on subspaces, that is used in the proof of \cref{cor:projformula}.
For any linear subspace $V \subseteq \bbR^n$, we denote $V^\perp$ the orthogonal space to $V$, that is $V^\perp = \{ w \in \bbR^n : w \cdot v = 0 \text{ for any } v \in V \}$.
Given a vector $p\in\bbR^n$ and a subspace $V$ of $\bbR^n$, we denote by $p_V$ the orthogonal projection of $p$ on $V$.

\begin{lemma}\label{lmm:geomprojform}
  Suppose $V\subseteq\R^n$ is a linear subspace, $v\in V$ is a vector, and $C\subseteq\R^n$ satisfies
  $C\subseteq v+V^{\perp}$.
  If $\ V\cap C\neq\varnothing$, then
  \[
    (1)\ V\cap C=\{v\} \quad \text{and} \quad (2)\  \|v\|=\operatorname{dist}(0,C).
  \]
\end{lemma}
\begin{proof}
Note that for every $y \in V$, there holds $y_{V^\perp} = 0$ and
\begin{equation}\label{eq:xV}
    y = y_V.
\end{equation}

(1)
Assume $w\in V\cap C$.
Since $w\in C \subseteq v + V^\perp$, there exists $z\in V^\perp$ such that $w = v + z$. As $v\in V$, using \eqref{eq:xV} we have $v = v_V$. Combining the previous two equalities gives $w=v_V+z$, which directly implies
\begin{equation}\label{eq:wV}
    w_V = v_V.
\end{equation}
Hence, as $v\in V$ and $w\in V$,
we have $w\overset{\eqref{eq:xV}}{=}w_V\overset{\eqref{eq:wV}}{=}v_V\overset{\eqref{eq:xV}}{=} v$. Consequently,
$V\cap C=\{v\}$.

(2) Let $x\in C$ be arbitrary and take $z\in V^{\perp}$ such that $x=v+z$. Note that:
\[
\|x\|^2 \ = \ \|v+z\|^2 \ = \ \|v\|^2+2(v\cdot z)+\|z\|^2 \ = \ \|v\|^2+\|z\|^2 \ \geq \ \|v\|^2.
\]
Thus $\operatorname{dist}(0,C)\geq \|v\|$. Since $V\cap C\neq \varnothing$, we have $v\in C$ by (1). Thus $\|v\|=\operatorname{dist}(0,C)$.
\end{proof}

\subsection{Differentiability and first-order approximations of continuous curves.}
\label{sec:appx_Frechet}
In this section, we recall the definition of (Fréchet) differentiability, and prove some basic facts on first-order approximations of nonsmooth continuous curves.
The main result is \cref{lmm:foapp}, used in the proof of \cref{lemma:chainrule}.

Throughout this section, we consider two functions $f,g:U\to\R$, where $U$ denotes an open subset of $\bbR^\ndim$, and a point $x_0\in U$.

\begin{definition}\label{def:foapp}
  We say that $g$ is a \textbf{first-order approximation} of $f$ at $x_0$ if:
  \[
    f(x_0)=g(x_0) \quad \text{and} \quad\lim_{x\to x_0}\frac{|f(x)-g(x)|}{|x-x_0|} \ = \ 0.
  \]
\end{definition}

\begin{remark}\label{rem:first_order_approx}
The notion ``$g$ is a first-order approximation of $f$ at $x_0$'' defines an equivalence relation on $\R$-valued functions defined on neighborhoods of $x_0$ in $U$.
\end{remark}

\begin{definition}\label{def:Frechet}
  We say that $f:U\to\bbR^m$ is \textbf{(Fr\'{e}chet) differentiable} at $x_0$ if there exists a linear map $\linearmap:\R^\ndim\to\R^m$ such that $f(x_0)+\linearmap (x-x_0)$ is a first-order approximation of $f$ at $x_0$.
  In this case, we define the \textbf{Jacobian} $J_f(x_0)$ of $f$ at $x_0$ as the matrix of the linear mapping $\linearmap$, so that $J_f \cdot d = \linearmap(d)$ for every $d\in\R^\ndim$.
  When $m=1$, we rather consider the \textbf{gradient} $\nabla f(x_0)$ of $f$ at $x_0$, defined as the unique vector in $\R^\ndim$ such that $\nabla f(x_0)^Td=\linearmap (d)$ for every $d\in\R^\ndim$.
\end{definition}

Note that, in particular, $f$ is a first-order approximation of $g$ at $x_0$ as soon as $f$ and $g$ are differentiable at $x_0$, and $f'(x_0)=g'(x_0)$.

We are now ready to state the main result of this section.
\begin{lemma}\label{lmm:foapp}
  Consider $t_0\in\R$ and $\varepsilon>0$, two curves $x,\gamma:(t_0 -\varepsilon,t_0 + \varepsilon)\to\R^\ndim$, and a function $f:\bbR^\ndim\to\bbR$. If
    $x$ and $\gamma$ are continuous at $t_0$,
    $x$ is a first-order approximation of $\gamma$ at $t_0$, and
    $f$ is locally Lipschitz at $x_0:=x(t_0)=\gamma(t_0)$,
  then
  \begin{enumerate}
    \item $f\circ x$ is a first-order approximation of $f\circ\gamma$ at $t_0$ (and thus also vice-versa); \label{it:firstorderapp}
    \item the following are equivalent: \label{it:Fdiff}
      \begin{enumerate}
        \item $f\circ x$ is differentiable at $t_0$,
        \item $f\circ \gamma$ is differentiable at $t_0$.
      \end{enumerate}
      Moreover, if either of the above holds, then $(f\circ x)'(t_0)=(f\circ\gamma)'(t_0)$.
  \end{enumerate}
\end{lemma}
\begin{proof}
  (\ref{it:firstorderapp}) Let $L$ be a Lipschitz modulus of $f$ around $x_0$. Then for $t$ sufficiently close to $t_0$, we have: $x(t)$ and $\gamma(t)$ are sufficiently close to $x_0$ (by continuity)
  and thus we have:
  \[
    \frac{|f(x(t))-f(\gamma(t))|}{|t-t_0|} \ \leq \ \frac{L|x(t)-\gamma(t)|}{|t-t_0|} \ \to \ 0 \quad\text{as} \quad t \ \to \ t_0
  \]
  since $x$ is a first-order approximation of $\gamma$ at $t_0$.

  (\ref{it:Fdiff}) If either (a) or (b) holds, then there is a linear map $\linearmap:\R\to\R$ such that $f(x_0)+\linearmap(t-t_0)$ is a first-order approximation of both $f\circ x$ and $f\circ \gamma$ at $t_0$ (using (1) and \cref{rem:first_order_approx}). This linear map is the derivative of both functions at $t_0$.
\end{proof}

\section*{Acknowledgments}
This work has received funding from the European Union’s Horizon Europe research and innovation programme under grant agreement No. 10107056 (Human-compatible AI with Guarantees).

%% file: figures/stratificationS.tex
\begin{tikzpicture}[even odd rule]

   \clip (-3.4,-0.2) rectangle (4.5,2.2);

  \fill[pattern=dots, pattern color = c1] (-1,0) circle (1.2) (-1,0) circle (2);
  \fill[pattern=dots, pattern color = c1]  (1,0) circle (1.2)  (1,0) circle (2);

  \fill[pattern=dots, pattern color = c2] (3.5,0.5) rectangle (4,2);
  \fill[white] (-3,-0.6) rectangle (4,0); %

  \draw[c2, very thick] (-1,0) ++(0:2) arc (0:34.92:2); %
  \draw[c3, very thick] (-1,0) ++(60:2) arc (60:180:2); %
  \draw[c4, very thick] (-1,0) ++(0:1.2) arc (0:33.56:1.2); %
  \draw[c5, very thick] (-1,0) ++(72.54:1.2) arc (72.54:180:1.2); %
  \draw[c6, very thick] (1,0) ++(180:2) arc (180:145.08:2); %
  \draw[c7, very thick] (1,0) ++(120:2) arc (120:0:2); %
 \draw[c8, very thick] (1,0) ++(180:1.2) arc (180:146.44:1.2); %
 \draw[c9, very thick] (1,0) ++(107.46:1.2) arc (107.46:0:1.2); %
 \draw[c10, very thick] (-3,0.01) -- (-2.2,0.01);
 \draw[c11, very thick] (-1,0.01) -- (-0.2,0.01);
 \draw[c12, very thick] (1,0.01) -- (0.2,0.01);
 \draw[c13, very thick] (3,0.01) -- (2.2,0.01);

  \draw[c2, very thick] (3.5,0.5) -- (4,0.5);
  \draw[c4, very thick] (4,0.5) -- (4,2);
  \draw[c6, very thick] (3.5,2) -- (4,2);
  \draw[c10, very thick] (3.5,2) -- (3.5,0.5);

  \fill[black] (3.5,0.5) circle (1.5pt);
  \fill[black] (4,0.5) circle (1.5pt);
  \fill[black] (3.5,2) circle (1.5pt);
  \fill[black] (4,2) circle (1.5pt);

  \fill[black] (-3,0) circle (1.5pt);
  \fill[black] (-2.2,0) circle (1.5pt);
  \fill[black] (-1,0) circle (1.5pt);
  \fill[black] (-.2,0) circle (1.5pt);
  \fill[black] (.2,0) circle (1.5pt);
  \fill[black] (1,0) circle (1.5pt);
  \fill[black] (2.2,0) circle (1.5pt);
  \fill[black] (3,0) circle (1.5pt);
  \fill[black] (0,1.732) circle (1.5pt);
  \fill[black] (0,0.663) circle (1.5pt);
  \fill[black] (0.64,1.145) circle (1.5pt);
  \fill[black] (-0.64,1.145) circle (1.5pt);

  \node[fill=white,text=c1] at (0,1.2) {$I$};
  \node[text=c2] at (-2.5,1.7) {$E_1$};
  \node at (-3.25,0) {$P_1$};
  \node at (0,2) {$P_2$};
   
\end{tikzpicture}

%% file: figures/toposinecurve.tex
\begin{tikzpicture}
  \begin{axis}[
      trig format plots=rad,
      width=8cm,
      height=6cm,
      xmajorgrids,
      ymajorgrids,
      xmin=-0.04, xmax=1,
      ymin=-1.1, ymax=1.1,
      legend pos=south east,
      legend cell align=left,
    ]
    \addplot[
      tomato,
      no marks,
      samples=500,
      domain=-3:0,
    ] ({10^x}, {sin(1/(10^x))});
    \addlegendentry{$t \mapsto \sin t^{-1}$}
    \addplot[
      chartreuse,
      no marks,
      line width=2.5pt,
    ] coordinates {(0,-1) (0,1)};
    \addlegendentry{limit set}
  \end{axis}
\end{tikzpicture}

%% file: figures/o-min-structures.tex
\begin{tikzpicture}[
    yscale=0.8,
    every node/.style={inner sep=1pt},
    regionlabel/.style={text=c1},
    >={Latex[scale=1.5]},
    transform shape=false
  ]

  \node[font=\large] (Ralg)     at (0,    0)   {$\mathbb{R}_{\mathrm{alg}}$};
  \node[font=\large] (Rarctan)  at (3.8,  0)   {$\mathbb{R}_{\mathrm{arctan}}$};
  \node[font=\large] (Ran)      at (7.8,  0)   {$\mathbb{R}_{\mathrm{an}}$};
  \node[font=\large] (RG)       at (10.8, 0)   {$\mathbb{R}_{\mathcal{G}}$};
 
  \node[font=\large] (RalgR)    at (0,    3)   {$\mathbb{R}_{\mathrm{alg}}^{\mathbb{R}}$};
 
  \node[font=\large] (Rexp)     at (0,    6)   {$\mathbb{R}_{\mathrm{exp}}$};
  \node[font=\large] (Ranexp)   at (7.8,  6)   {$\mathbb{R}_{\mathrm{an,exp}}$};
  \node[font=\large] (RGexp)    at (10.8, 6)   {$\mathbb{R}_{\mathcal{G}\text{,exp}}$};
 
  \node[font=\large] (RPfaff)   at (2.2,  8.8) {$\mathbb{R}_{\mathrm{Pfaff}}$};
  \node[font=\large] (PfaffRan) at (7.8,  8.8) {$\mathrm{Pfaff}(\mathbb{R}_{\mathrm{an}})$};
  \node[font=\large] (PfaffRG)  at (10.8, 8.8) {$\mathrm{Pfaff}(\mathbb{R}_{\mathcal{G}})$};
 
  \node at (5.5, 11) {???};

  \draw[->] (Ralg.east)      -- (Rarctan.west);
  \draw[->] (Rarctan.east)   -- (Ran.west);
  \draw[->] (Ran.east)       -- (RG.west);
 
  \draw[->] (Rexp.east)      -- (Ranexp.west);
  \draw[->] (Ranexp.east)    -- (RGexp.west);
 
  \draw[->] (RPfaff.east)    -- (PfaffRan.west);
  \draw[->] (PfaffRan.east)  -- (PfaffRG.west);

  \draw[->] (Ralg.north)     -- (RalgR.south);
  \draw[->] (RalgR.north)    -- (Rexp.south);
 
  \draw[->] (Ran.north)      -- (Ranexp.south);
  \draw[->] (Ranexp.north)   -- (PfaffRan.south);
 
  \draw[->] (RG.north)       -- (RGexp.south);
  \draw[->] (RGexp.north)    -- (PfaffRG.south);

  \draw[<-] (RPfaff.south west) -- (Rexp.north east);
 
  \draw[<-] (RPfaff.south)  -- (Rarctan.north);

  \draw[dashed, line width=1.2pt, c1]
    (-1, 10) -- (11.5, 10);
 
  \draw[dashed, line width=1.2pt, c5]
    (-1, 4.5) -- (11.5, 4.5);

  \node[anchor=east] at (-2, 6)  {field of exponents $\mathbb{R}$};
  \node[anchor=east] at (-2, 0)  {field of exponents $\mathbb{Q}$};
 
  \draw[decorate, decoration={brace, amplitude=10pt}] (12.8,10) -- (12.8,0);
  \draw[decorate, decoration={brace, amplitude=10pt}] (12.8,12) -- (12.8,10.1);
  \draw[decorate, decoration={brace, amplitude=10pt}] (11.7,4.5) -- (11.7,0);
  \draw[decorate, decoration={brace, amplitude=10pt}] (-1.5,3) -- (-1.5,9);

  \node[regionlabel, anchor=west] at (13.3, 11)   {transexponential};
    \node[regionlabel, anchor=west, align=left] at (13.4, 5)  {exponentially \\ bounded};
  \node[regionlabel, anchor=west, rotate=-90, color=c5] at (12.5, 5)  {polynomially bounded};

\end{tikzpicture}

%% file: figures/spirallimit.tex
\begin{tikzpicture}
  \begin{axis}[
      trig format plots=rad,
      width=8cm,
      height=6cm,
      axis equal,
      xmajorgrids,
      ymajorgrids,
      legend pos=north east,
      legend cell align=left,
      legend style={fill=white, fill opacity=0.85, text opacity=1, draw=black},
    ]
    \addplot[
      chartreuse,
      no marks,
      line width=2.5pt,
      samples=100,
      domain=0:2*pi,
    ] ({cos(x)}, {sin(x)});
    \addlegendentry{limit set}
    \addplot[
      tomato,
      no marks,
      samples=200,
      domain=1:50,
    ] ({(1+1/x)*sin(x)}, {(1+1/x)*cos(x)});
    \addlegendentry{spiral}
  \end{axis}
\end{tikzpicture}

%% file: figures/3dcurve.tex
\begin{tikzpicture}
  \begin{axis}[
      trig format plots=rad,
      width=8cm,
      height=8cm,
      axis equal image,
      view={45}{10},
      xmin=-1.1, xmax=1.1,
      ymin=-1.1, ymax=1.1,
      zmin=-1.1, zmax=1.1,
      legend pos=north west,
      legend cell align=left,
      legend style={fill=white, fill opacity=0.85, text opacity=1, draw=black},
      colormap={chart}{color=(chartreuse) color=(chartreuse)},
    ]
    \addplot3[
      surf,
      shader=flat,
      opacity=0.25,
      draw=chartreuse,
      samples=24,
      samples y=24,
      domain=0:2*pi,
      y domain=0:pi,
      z buffer=sort,
    ] ({cos(x)*sin(y)}, {sin(x)*sin(y)}, {cos(y)});
    \addlegendentry{limit set}
    \addplot3[
      tomato,
      no marks,
      samples=600,
      samples y=0,
      domain=1:29*pi,
    ] (
      {(1+1/x)*sin(x)*cos(sqrt(2)*x)},
      {(1+1/x)*sin(x)*sin(sqrt(2)*x)},
      {(1+1/x)*cos(x)}
    );
    \addlegendentry{curve}
  \end{axis}
\end{tikzpicture}

%% file: figures/whitneycusp.tex
\resizebox{\linewidth}{!}{
\begin{tikzpicture}
  \begin{axis}[
      width=10cm,
      height=8cm,
      view={145}{12},
      xmin=-1.05, xmax=1.05,
      ymin=-1.05, ymax=1.05,
      zmin=-1.45, zmax=1.05,
      xtick=\empty, ytick=\empty, ztick=\empty,
      xlabel=$x$, ylabel=$y$, zlabel=$z$,
      legend pos=outer north east,
      legend cell align=left,
      colormap={cuspmap}{color=(tomato) color=(tomato)},
    ]
    \addplot3[
      surf,
      shader=flat,
      opacity=0.6,
      faceted color=tomato!55!black,
      samples=45,
      samples y=45,
      domain=-1:1,
      y domain=-1.1:1.1,
      z buffer=sort,
    ] ({x}, {y*(x^2-y^2)}, {x^2-y^2});
    \addlegendentry{$M_2$ (smooth surface)}
    \addplot3[
      chartreuse,
      no marks,
      line width=2.5pt,
    ] coordinates {(-1,0,0) (1,0,0)};
    \addlegendentry{$M_1'$ (punctured $x$-axis)}
    \addplot3[
      only marks,
      mark=*,
      mark size=2.5pt,
      black,
    ] coordinates {(0,0,0)};
    \addlegendentry{$M_0$ (origin)}
  \end{axis}
\end{tikzpicture}
}